\documentclass[english, a4paper, 12pt,reqno]{amsart}
\usepackage[utf8]{inputenc}
\usepackage{multicol,amsmath,t1enc,a4,babel}
\usepackage{amsthm}
\usepackage{amssymb}
\usepackage[dvips]{graphicx}
\usepackage[dvips]{epsfig}
\usepackage{times}
\usepackage{amsfonts}
\usepackage{wasysym}
\usepackage{mathtools}
\usepackage[colorinlistoftodos]{todonotes}
\usepackage{mathrsfs}
\usepackage{relsize}
\usepackage{bbm}
\RequirePackage{doi}
\usepackage{hyperref}
\usepackage{enumitem}

\newtheorem{defi}{Definition}[section]
\newtheorem{lemma}[defi]{Lemma}

\newtheorem{theorem}[defi]{Theorem}
\newtheorem{kor}[defi]{Corollary}

\newtheorem{rem}[defi]{Remark}
\newtheorem{ex}[defi]{Example}

\begin{document}
\title[Characterization of continuous stationary fields]{Characterization of continuous stationary fields as generalized Ornstein-Uhlenbeck fields via multi-parameter Langevin equation and multiple Riemann-Stieltjes integration}

	\author[Voutilainen]{Marko Voutilainen}
	\address{Marko Voutilainen,
		Department of Accounting and Finance,
		University of Turku,
		20014 Turun yliopisto,
		Finland}
		\email{marko.voutilainen@utu.fi}
	\author[Ilmonen]{Pauliina Ilmonen}
	\address{Pauliina Ilmonen,
		Department of Mathematics and Systems Analysis,
		Aalto University School of Science,
		00076 Aalto,
		Finland}
		\email{pauliina.ilmonen@aalto.fi}
	\author[Viitasaari]{Lauri Viitasaari}
	\address{Lauri Viitasaari,
		Department of Information and Service Management, Aalto University School of Business, 00076 Aalto, Finland}
		\email{lauri.viitasaari@aalto.fi}

		\keywords{Generalized Ornstein-Uhlenbeck fields, Langevin equations, Self-similar fields, Stationary fields, Stationary increment fields, Multiple Riemann-Stieltjes integration}
		\subjclass[2020]{60G05, 60G60, 60G10, 60G18, 26B99}
		\date{\today}
		
		\begin{abstract}
			In this article, we characterize continuous stationary fields via generalized Langevin dynamics. This gives natural connections between stationary fields, stationary increment fields, self-similar fields, and generalized Langevin dynamics. Our contribution extends some recently proved similar results for stochastic processes to the case of continuous random fields. As a by-product, we introduce some new results on multiple Riemann-Stieltjes integrals.
		
		\end{abstract}

\maketitle

\section{Introduction}
Random phenomena evolving in time are often modeled using stochastic processes $X=\{X(t)\}_{t\in T}$, where $T$ is a subset of the one-dimensional Euclidean space, and the role of modeling is to assess the statistical characteristics of the process $X$. Important special classes of such processes include stationary processes, self-similar processes, and processes with stationary increments, as each of these classes allows powerful tools for modeling and understanding random phenomena across diverse fields. Indeed, stationary processes model phenomena where statistical properties remain unchanged over time, self-similar processes model phenomena where statistical characteristics are consistent across different scales (for details on self-similar processes, see the monograph \cite{embrechts2002selfsimilar}), and stationary increment processes are particularly suitable models, for example, when considering random movements.

Connections between the three classes, stationary processes, self-similar processes, and processes with stationary increments, are already rather well-understood in the literature. For the first results relating different classes, Lamperti proved in \cite{Lamperti-1962} that there is a one-to-one correspondence between stationary processes and self-similar processes, and the underlying bijection is nowadays known as the Lamperti transformation. Later, stationary increment processes were connected in \cite{Viitasaari-2016a}, where (continuous) stationary processes in continuous time were characterized as solutions to the Langevin equation 
\begin{equation}
\label{eq:introlang}
dX(t) = - \theta X(t) dt + dG(t),
\end{equation}
where $G$ is a suitable process with stationary increments, and the stochastic differential equation can be understood via integration in the Riemann-Stieltjes sense. Equation \eqref{eq:introlang} gives rise to generalized Ornstein-Uhlenbeck processes, links all three classes intimately together, and one can rather freely move from one class to another with suitable transformations. Such results have later been extended to discrete time \cite{voutilainen2017model} and to vector-valued processes both in continuous time \cite{voutilainen2021vector} and in discrete time \cite{voutilainen2020modeling}.

The features mentioned above, such as consistency of statistical properties over time or across different scales naturally appear also in the case of random fields, where the parameter $t$ of the process $\{X(t)\}_{t\in T}$ is multivariate. In this context, stationarity is still naturally defined, whereas the concepts of self-similarity and stationarity of increments are more complex. Increments of random fields are typically defined through rectangular increments, which capture changes related to all directions of the parameter space $T$ simultaneously. For literature related to stationarity of (rectangular) increments, see, e.g., \cite{makogin2015example,makogin2019gaussian} and the references therein. A version of the Lamperti theorem for random fields was established in \cite{genton2007self}, with self-similarity defined componentwise, allowing anisotropic scaling properties required in certain applications. For more details, see e.g. \cite{bierme2007operator}, where the authors adapt an alternative approach to anisotropy via the so-called operator scaling stable random fields.
For results connecting all three classes of fields via Langevin-type dynamics, we can only mention \cite{discrete} studying fields indexed by $\mathbb{N}^N$ and taking values in $\mathbb{R}$, and later extensions in \cite{voutilainen2025one} to cover fields indexed by $\mathbb{N}^N$ and taking values in $\mathbb{R}^n$. The main results of these articles provide natural dynamics, similar to the process case \cite{voutilainen2017model}, linking all three classes.

In this article, we extend the results of \cite{discrete} to fields $\{X(t)\}_{t\in \mathbb{R}^N}$. Compared to \cite{discrete}, the main difficulty is that differences are replaced by differentials, and the notion of Riemann-Stieltjes integrals is much more subtle in the multivariate case. In particular, for our purpose, we need to extend the notion of multiple Riemann-Stieltjes integrals over non-hyperrectangles that are non-classical domains. Some of our main contributions include several results on multiple Riemann-Stieltjes integrals, which are particularly well-suited to our framework. On top of that, we provide an approach to define multiple Riemann-Stieltjes integrals over certain non-classical domains in situations where one cannot simply apply zero-extensions of functions and use a more classical approach. With this at hand, we provide natural transformations between stationary fields, fields with stationary rectangular increments, and self-similar fields, cf. Theorem \ref{theorem:lamperti}, Theorem \ref{theorem:bijection}, and Corollary \ref{cor:bijection}. We introduce a generalization of the Langevin dynamics \eqref{eq:introlang} (Definition \ref{defi:langevin}) and, through the transformations, characterize stationary fields as solutions to these dynamics; cf. Theorem \ref{theorem:1} and Theorem \ref{theorem:2} (see also Theorem \ref{theorem:supreme}), providing a natural extension of the main result of \cite{Viitasaari-2016a} from random processes to random fields. We also note that when the generalization of \eqref{eq:introlang} is driven by anisotropic fractional Brownian sheets, the resulting fields can naturally be termed fractional Ornstein–Uhlenbeck fields (of the first kind).

Let us also briefly discuss the existing literature on multiple Riemann-Stieltjes integration. In the univariate case, it is well-known that the Riemann-Stieltjes integral $\int f(s)dg(s)$ exists provided that one of the functions $f$ and $g$ is continuous and the other is of bounded variation. Multiple Riemann-Stieltjes integrals (over standard hyperrectangles) are well-studied in the literature, with a particular focus on the bivariate case. For details, we refer to the monographs \cite{sard1963linear,hildebrandt1963introduction}. In the multivariate case, the assumption on bounded variation is commonly replaced by bounded variation in the sense of Hardy-Krause, originally introduced in \cite{hardy1906double}. Different variations in the bivariate case are studied in \cite{clarkson1933definitions}. We also cite more recent works \cite{owen2005multidimensional,prause2015sequential} considering the general multivariate setting.

The remainder of the article is organized as follows. We begin with necessary preliminaries on random fields and on multiple Riemann-Stieltjes integrals in Section \ref{sec:preliminaries}. In Section \ref{sec:RS} we provide our results related to multiple Riemann-Stieltjes integrals. In particular, we provide an approach to define integrals over hypertriangles that is suitable for our purposes. This allows us to study connections between the three classes of random fields and the generalized Langevin equation in Section \ref{sec:characterisation}. 

\section{Preliminaries}
\label{sec:preliminaries}
By $\Theta\in(0,\infty)^N$, we denote an $N$-dimensional parameter vector $\Theta = (\theta_1,\dots, \theta_N)$, whose elements belong to the interval $(0,\infty)$. The cardinality of a set $u\subseteq \{1,\dots,N\}$ is denoted by $|u|$. The complement of $u$ with respect to $\{1,\dots,N\}$ is denoted by $-u$ and the complement of $u$ with respect to some other set $v\subseteq \{1,\dots,N\}$ is denoted by $v-u$. For the sake of clarity, the elements of such sets are ordered in ascending order. Let $t= (t_1,\dots,t_N)\in\mathbb{R}^N$ and $u\subseteq \{1,\dots,N\}$. We use the notation $t_u$ to pick the elements of $t$ whose indices belong to $u$. If also $s\in\mathbb{R}^N$, then by $s_u:t_{-u}$ we denote an $N$-vector $y$, for which $y_i = s_i$ if $i\in u$ and $y_i = t_i$ if $i\notin u$. By applying this notation, we can also pick elements from more than two vectors. Moreover, the ordering of our sets allows us to write, for example, $t_{u_v}$, where $v \subseteq \{1,\dots,|u|\}.$

\begin{ex}
Let $N = 5$, $u = \{2,4,5\}$, $v = \{2,3\}$ and $w = \{3\}.$ Then
$$u-v = \{4,5\},\ t_u = (t_2,t_4,t_5),\ (t_u:s_w:r_{-u-w}) = (r_1,t_2,s_3,t_4,t_5),$$
and 
$$t_{u_v} = t_{\{4,5\}} = (t_4,t_5) = (t_2,t_4,t_5)_v = (t_u)_v.$$
Moreover, in the case of single-element sets, we apply the notation familiar from the context of vectors:
$$t_{u_w} = t_{u_{\{3\}}}= t_{u_3} = t_5.$$
\end{ex}

The notation $[s,t] \coloneqq \left\{x\in\mathbb{R}^N : s_m \leq x_m \leq t_m, m \in \{1,\dots,N\}\right\}$ stands for $N$-dimensional hyperrectangles. The $|u|$-dimensional subrectangles $[s_u,t_u]$ can be defined accordingly. The mixed partial derivative of an $N$-variate function $f$ with respect to the coordinates in $u$ is denoted by $f_{t_u}$ and $\frac{df}{dt_u}$. Hereby, $f_t$ and $\frac{df}{dt}$ represent the mixed partial derivative, where $f$ has been differentiated once with respect to every coordinate. Moreover, we define the bracket notation $[f(r)]_s^t$ as
\begin{equation*}
 [f]_s^t= \left[f(r)\right]_s^t \coloneqq \sum_{v\subseteq\{1,\dots,N\}} (-1)^{|v|} f(s_v:t_{-v}). 
\end{equation*}
The bracket $\left[f(r)\right]_s^t$ corresponds to the rectangular increment of $f$ given by the points $s$ and $t$. Similarly, 
\begin{equation*}
\left[f(r)\right]_{s_u}^{t_u}\coloneqq \sum_{v\subseteq u} (-1)^{|v|} f(r_{-u}:s_v:t_{u-v})
\end{equation*}
corresponds to the increment of $f$ on the subrectangle $[s_u,t_u]$. Note that $\left[f(r)\right]_{s_u}^{t_u}$ is a function of $N - |u|$ variables $r_{-u}$.
We will return to this in Definition \ref{defi:increments}. Lastly, we present an elementary lemma on sums of binomial coefficients that we are going to utilize on several occasions. A proof can be found for example in \cite{discrete}.

\begin{lemma}
\label{binomial}
Let $M\in\mathbb{N}^+$. Then
\begin{equation*}
\sum_{m=0}^M (-1)^m\binom{M}{m} = 0.
\end{equation*}
\end{lemma}

\subsection{Stationarity, self-similarity, and stationarity of increments}
Next, we present the definitions of the three classes of random fields, together with the associated Lamperti theorem.
\begin{defi}
\label{defi:stationarity}
A random field $X = \{X(t)\}_{t\in\mathbb{R}^N}$ is stationary if 
$$ \{X(t+s)\}_{t\in\mathbb{R}^N} \overset{\text{law}}{=} \{X(t)\}_{t\in\mathbb{R}^N}$$
for every $s\in\mathbb{R}^N$ in the sense of finite-dimensional distributions.
\end{defi}

\begin{defi}
\label{defi:self-similarity}
Let $\Theta\in (0,\infty)^N$ and let 
$$Y = \{Y(e^t)\}_{t\in\mathbb{R}^N} = \{Y(e^{t_1}, \dots, e^{t_N})\}_{(t_1,\dots,t_N)\in\mathbb{R}^N}$$
be a random field. If
$$\{Y(e^{t+s})\}_{t\in\mathbb{R}^N} \overset{\text{law}}{=} \{e^{\Theta\cdot s}Y(e^{t})\}_{t\in\mathbb{R}^N}$$
for every $s\in\mathbb{R}^N$, then $Y$ is $\Theta$-self-similar.
\end{defi}

\begin{defi}
\label{defi:lamperti}
Let $\Theta\in (0,\infty)^N$, and let $X=\{X(t)\}_{t\in\mathbb{R}^N}$ and $Y=\{Y(e^t)\}_{t\in\mathbb{R}^N}$ be random fields. The Lamperti transformation $\mathcal{L}_\Theta$ and its inverse $\mathcal{L}^{-1}_\Theta$ are defined by
\begin{align*}
(\mathcal{L}_\Theta X)(e^t) &= e^{\Theta\cdot t} X(t), \quad t\in\mathbb{R}^N, \quad \text{and}\\
(\mathcal{L}^{-1}_\Theta Y)(t) &= e^{-\Theta\cdot t} Y(e^t), \quad  t\in\mathbb{R}^N.
\end{align*}
\end{defi}

\begin{rem}
By a simple change of variables, Definitions \ref{defi:self-similarity} and \ref{defi:lamperti} can be shown to be equivalent to the corresponding definitions in \cite{genton2007self}.  
\end{rem}
For a proof of the next theorem, we refer to \cite{genton2007self}.
\begin{theorem}
\label{theorem:lamperti}
Let $\Theta\in (0,\infty)^N$. If $X=\{X(t)\}_{t\in\mathbb{R}^N}$ is stationary, then $\mathcal{L}_\Theta X$ is $\Theta$-self-similar. Conversely, if $Y=\{Y(e^t)\}_{t\in\mathbb{R}^N}$ is $\Theta$-self-similar, then $\mathcal{L}^{-1}_\Theta Y$ is stationary.
\end{theorem}

\begin{defi}
\label{defi:increments}
Let $X = \{X(t)\}_{t\in\mathbb{R}^N}$ be a random field. Let $t,s\in\mathbb{R}^N$ be such that $s_l \leq t_l$ for every $1\leq l \leq N$. Then, the corresponding rectangular increment is
\begin{equation}
\label{increments}
\begin{split}
\left[X\right]_s^t &= \left[X(r)\right]_s^t = \sum_{v\subseteq\{1,\dots,N\}} (-1)^{|v|} X(s_v:t_{-v}) \\
 &= \sum_{(i_1,\dots,i_N)\in\{0,1\}^N} (-1)^{\sum_{l=1}^N i_l} X(t_1-i_1(t_1-s_1),\dots,t_N-i_N(t_N-s_N)).\\
 \end{split}
 \end{equation}
\end{defi}


\begin{rem}
\label{rem:degeneracy}
Degenerate increments, i.e. increments with $t_l = s_l$ for some $l$, equal to zero. 
\end{rem}

\begin{rem}
\label{rem:extension}
Definition \ref{defi:increments} extends in a straightforward and intuitive way to cases where restriction $t \geq s$ does not hold. Let $m$ be the number of indices $l$ for which $t_l < s_l$ and let $\hat{t}$ and $\hat{s}$ denote the vectors, where elements $t_l$ and $s_l$ have been swapped so that $\hat{t} \geq \hat{s}$ is satisfied. Then
$[X]_s^t = (-1)^m [X]_{\hat{s}}^{\hat{t}}.$
\end{rem}

\begin{defi}
\label{defi:stationaryincrements}
Let $X = \{X(t)\}_{t\in\mathbb{R}^N}$ be a random field. Let $n\in\mathbb{N}$, and let 
$$h, t^{(1)}, \dots, t^{(n)}, s^{(1)}, \dots, s^{(n)} \in \mathbb{R}^N$$
be arbitrary. If 
\begin{equation*}
\begin{pmatrix}
[X]_{s^{(1)}+h}^{t^{(1)}+h}\\
\vdots\\
[X]_{s^{(n)}+h}^{t^{(n)}+h}
\end{pmatrix} \overset{\text{law}}{=} \begin{pmatrix}
[X]_{s^{(1)}}^{t^{(1)}}\\
\vdots\\
[X]_{s^{(n)}}^{t^{(n)}}
\end{pmatrix},
\end{equation*}
then $X$ is a stationary increment field.
\end{defi}

\begin{rem}
\label{rem:equivalent}
It can be shown that Definition \ref{defi:stationaryincrements} is equivalent to 
\begin{equation}
\label{alternative}
\{ [X]_h^{t+h}\}_{t\in\mathbb{R}^N} \overset{\text{law}}{=} \{[X]_0^t\}_{t\in\mathbb{R}^N} \quad\text{for all } h\in\mathbb{R}^N.
\end{equation}
Furthermore, \eqref{alternative} is the adapted version of the definition of stationarity of increments discussed, e.g., in \cite{makogin2019gaussian}.
\end{rem}
So far, we have not posed assumptions related to the regularity of the sample functions of our fields. As we next proceed to multiple Riemann-Stieltjes integration, we make a standing assumption that the realizations are almost surely continuous. This ensures the existence of appropriate integrals in the given sense.

\subsection{Multiple Riemann-Stieltjes integrals}

Assume that $N$-variate functions $f$ and $g$ are defined on hyperrectangle $[s,t]$, where $s,t\in\mathbb{R}^N$ with $s\leq t$ component-wise. Let $s_m = x_0^{(m)} < x_1^{(m)} < \dots < x_{n_m}^{(m)}= t_m$ be a partition of the one-dimensional interval $[s_m,t_m]$, where $m\in\{1,\dots,N\}$. The norm of such a partition $P_m$ is 
$$\Vert P_m \Vert = \max_{1 \leq i \leq n_m}\{|x^{(m)}_i - x^{(m)}_{i-1}|\}.$$  
The one-dimensional partitions define a division of the rectangle $[s,t]$ into subrectangles or cells
$$I_i = I_{i_1,\dots, i_N} \coloneqq [(x^{(1)}_{i_1-1}, \dots, x^{(N)}_{i_N-1}),(x^{(1)}_{i_1}, \dots, x^{(N)}_{i_N})], $$
where $i_m\in\{1,\dots,n_m\}$.
The norm of this rectangular partition $P$ is 
$$\Vert P \Vert = \max_{1\leq m \leq N} \{\Vert P_m \Vert\}.$$
Let $\xi_i = \xi_{i_1,\dots,i_N} = (\xi^{(1)}_{i_1,\dots,i_N}, \dots, \xi^{(N)}_{i_1,\dots,i_N})$ be any $N$-vectors satisfying $\xi_i \in I_i$ that we call tags. If
\begin{equation}
\label{rssum}
\lim_{\Vert P \Vert\to 0} \sum_{i_1 = 1}^{n_1}\dots\sum_{i_N = 1}^{n_N} g(\xi_i)[f]_{(x^{(1)}_{i_1-1}, \dots, x^{(N)}_{i_N-1})}^{(x^{(1)}_{i_1}, \dots, x^{(N)}_{i_N})}
\end{equation}
exists as a finite number, then we call it the (unrestricted) Riemann-Stieltjes integral of $g$ with respect to $f$. We note that if the 
tags are of the form $\xi_i = (\xi^{(1)}_{i_1}, \dots,\xi^{(N)}_{i_N}) $, then the integral is known as the restricted Riemann-Stieltjes integral in the literature (for details in the bivariate case, see e.g. \cite{sard1963linear}). Throughout, in our setting, we consider unrestricted Riemann-Stieltjes integrals. Moreover, we denote the integral \eqref{rssum} as
\begin{equation*}
\int_{s_1}^{t_1} \dots \int_{s_N}^{t_N} g(u)   d_{\{1,\dots,N\}}f(u) = \int_s^t g(u) df(u) = \int_s^t g df.
\end{equation*}
Note that due to the alternating sign in \eqref{increments},
$$\int_s^t df = [f]_s^t$$
for all functions $f$.
Assume that a partition $P$ is fixed and let $\Delta_j$ be the difference operator in coordinate $j$ defined as
\begin{equation*}
\begin{split}
\Delta_j f(x^{(1)}_{i_1}, \dots, x^{(N)}_{i_N})
&= f(x^{(1)}_{i_1}, \dots, x^{(j-1)}_{i_{j-1}}, x^{(j)}_{i_j}, x^{(j+1)}_{i_{j+1}},\dots, x^{(N)}_{i_N})\\
&\quad - f(x^{(1)}_{i_1}, \dots,x^{(j-1)}_{i_{j-1}}, x^{(j)}_{i_j-1}, x^{(j+1)}_{i_{j+1}},\dots, x^{(N)}_{i_N}).
\end{split}
\end{equation*}
Then
\begin{equation}
\label{difference}
[f]_{(x^{(1)}_{i_1-1}, \dots, x^{(N)}_{i_N-1})}^{(x^{(1)}_{i_1}, \dots, x^{(N)}_{i_N})} = \Delta_1\dots \Delta_N f(x^{(1)}_{i_1}, \dots, x^{(N)}_{i_N})\eqqcolon \Delta f(x_i).
\end{equation}
For $v \subseteq\{1,\dots,N\}$, we extend previously set notation to comprehend $d_vf$ and $du_v$. The former denotes the multivariate Riemann-Stieltjes differential over the coordinates in $v$ and the latter denotes the usual multivariate Riemann or Lebesgue differential over the coordinates in $v$. Naturally, these notations are applied in the relation of integration. For example,
\begin{equation*}
\int_{s_v}^{t_v} [g(r)d_v f(r)]_{s_{-v}}^{t_{-v}}
\end{equation*}
stands for evaluating the increment of $g(r) d_vf(r)$ on the subrectangle $[s_{-v}, t_{-v}]$. This results in $2^{|-v|}$ $|v|$-variate Riemann-Stieltjes integrals of $g$ with respect to $f$ over the coordinates $r_v$.

Next, we consider variations of multivariate functions. Let $P$ be a partition of the interval $[s,t]$ as described above. The variation of $f$ over $P$ is
$$ V_P(f) \coloneqq \sum_{i_1 = 1}^{n_1}\dots\sum_{i_N = 1}^{n_N}\left| [f]_{(x^{(1)}_{i_1-1}, \dots, x^{(N)}_{i_N-1})}^{(x^{(1)}_{i_1}, \dots, x^{(N)}_{i_N})} \right|.$$
We denote the set of all possible partitions of a fixed interval $[s,t]$ with $\mathcal{P}$.

\begin{defi}
\label{defi:vitali}
Let $f$ be defined on $[s,t]$. The variation of $f$ on $[s,t]$ in the sense of Vitali is
\begin{equation*}
V(f, [s,t]) \coloneqq  \sup_{P\in\mathcal{P}} V_P(f).
\end{equation*}
If $V(f, [s,t]) < \infty$, we denote $f\in BV[s,t]$.
\end{defi}

\begin{defi}
\label{defi:HKBV}
Let $f$ be defined on $[s,t]$. The variation of $f$ on $[s,t]$ in the sense of Hardy and Krause is
\begin{equation*}
V_{HK}(f, [s,t]) \coloneqq \sum_{\emptyset \neq v \subseteq \{1,\dots,N\}} V (f(r_v:t_{-v}), [s_v,t_v]).
\end{equation*}
If $V_{HK}(f, [s,t]) < \infty$, we denote $f\in BVHK[s,t]$.
\end{defi}

\begin{rem}
\label{rem:variations}
If the rectangular domain $[s,t]$ is clear from the context, we may simply write $f\in BV$ and $f\in BVHK$. The variation in the sense of Hardy and Krause is the sum of variations in the sense of Vitali over all subrectangles $[s_v,t_v]$, where the coordinates $-v$ have been fixed to $t_{-v}$. In the one-dimensional setting, the two variations are equal to the usual variation of univariate functions. 
Moreover, in Definition \ref{defi:HKBV}, replacing $t_{-v}$ with any fixed point from the subrectangle $[s_{-v}, t_{-v}]$, would not change the class $BVHK$ of functions. See \cite{owen2005multidimensional}, \cite{sard1963linear}, and \cite{youngfourier}.
\end{rem}

For details on the next lemma, see, e.g., \cite{owen2005multidimensional} and the  references therein.
\begin{lemma}
\label{lemma:bvkhproduct}
Let $f,g \in BVHK[s,t]$. Then $f+g$ and $fg$ are both in $BVHK[s,t]$.
\end{lemma}

The next lemma shows that $N$ times continuously differentiable functions are bounded variations in the sense of Hardy and Krause. For a proof, we refer to \cite{owen2005multidimensional}.
\begin{lemma}
\label{lemma:variationofdifferentiable}
Let $f$ be a function such that $f_t$ exists and is continuous on $[s,t]$. Then
\begin{equation*}
V_{HK}(f, [s,t]) = \sum_{\emptyset\neq v\subseteq \{1,\dots,N\}} \int_{s_v}^{t_v} \left|f_{t_v}(r_v:t_{-v})\right| dr_v.
\end{equation*}
\end{lemma}
\begin{lemma}
\label{lemma:existence}
Let $g$ be a continuous function on $[s,t]$ and let $f$ be a bounded variation in the sense of Hardy and Krause on $[s,t]$. Then 
\begin{equation}
\label{subintegrals}
\int_{s_v}^{t_v}[g(r) d_vf(r)]_{s_{-v}}^{t_{-v}}
\end{equation}
exists for all $\emptyset \neq v \subseteq \{1,\dots,N\}$.
\end{lemma}
\begin{proof}
It was shown in \cite{clarkson1933double} that the unrestricted integral $\int_s^t gdf$ exists for all continuous $g$ if and only if $f\in BV$ (See also \cite{sard1963linear}). While the result was proved only in the bivariate case, it is a routine book-keeping exercise (noted in \cite{clarkson1933double}) to obtain the claim in arbitrary dimensions. Hence our claim follows directly by the very definition of $f$ being bounded variation in the sense of Hardy and Krause. Indeed, if $f$ is bounded variation in the sense of Hardy and Krause, it implies that all lower dimensional projections (for which some variables are fixed) are of bounded variation as well.
\end{proof}
\begin{rem}
\label{rem:existence2}
For the existence of the integral \eqref{subintegrals} over the whole hyperrectangle, i.e. $v = \{1,\dots,N\}$, a sufficient condition is $f\in BV[s,t]$. A proof is presented in \cite{prause2015sequential}.
\end{rem}

The following integration by parts formula was originally presented in \cite{young1917multiple}.
\begin{lemma}
\label{lemma:ipp}
Let $f$ and $g$ be functions defined on $[s,t]$. If
\begin{equation*}
\int_{s_v}^{t_v}[g(r) d_vf(r)]_{s_{-v}}^{t_{-v}}
\end{equation*}
exists for all $\emptyset \neq v \subseteq \{1,\dots,N\}$, then also
\begin{equation*}
\int_{s_v}^{t_v}[f(r) d_vg(r)]_{s_{-v}}^{t_{-v}}
\end{equation*}
exists for all $\emptyset \neq v \subseteq \{1,\dots,N\}$. In addition, the integration by parts formula 
\begin{equation}
\label{iip}
\int_s^t f(r) dg(r) = [f(r)g(r)]_s^t + \sum_{\emptyset\neq v \subseteq \{1,\dots,N\}}(-1)^{|v|} \int_{s_v}^{t_v} [g(r)d_vf(r)]_{s_{-v}}^{t_{-v}}
\end{equation}
holds.
\end{lemma}


\begin{defi}[Langevin equation]
\label{defi:langevin}
Let $\Theta \in(0, \infty)^{N}$ and let $G = \{G(t)\}_{t\in\mathbb{R}^N}$. The corresponding Langevin differential equation is
\begin{equation}
\label{langevinequation}
\sum_{u\subseteq\{1,\dots,N\}} \prod _{i\in u}\theta_i d_{-u}X(t) dt_u = dG(t).
\end{equation}

\end{defi}

\begin{ex}
\label{ex:2D}
In the case $N=2$, the Langevin differential equation is
$$d_{\{1,2\}} X(t)+\theta_{1} d_{2} X(t) d t_{1}+\theta_{2} d_{1} X(t) d t_{2}+\theta_1\theta_2 X(t) d t_{1} d t_{2}=d_{\{1,2\}} G(t).$$
\end{ex}

\begin{rem}
\label{rem:novel}
Definition \ref{defi:langevin} is a generalization of the Langevin equation to cover arbitrary $N$. For the case $N=2$, similar Langevin equations can be found at least from \cite{de2012least} and \cite{terdik2005notes}. In \cite{de2012least} the parameter is a scalar and the equation does not contain mixed differentials of type $d_{t_1}X(t)dt_2$, while in \cite{terdik2005notes} the authors consider 
the same equation as us except that the signs of the mixed terms are negative.
\end{rem}

\begin{rem}
\label{rem:interpretation}
We understand Equation \eqref{langevinequation} in the multivariate Riemann-Stieltjes sense. That is, the equation holds when both sides are integrated from 0 to $t$ with $t\in\mathbb{R}^N$, and the integrals with stochastic integrators are interpreted as multiple Riemann-Stieltjes integrals.
\end{rem}

\begin{ex}
\label{ex:3D}
In the 3-dimensional case, the term $d_{\{2,3\}} X(t) dt_1$ gives
$$\int_0^t d_{\{2,3\}} X(s) ds_1= \int_0^{t_1} \int_0^{t_2} \int_0^{t_3}  d_{\{2,3\}} X(s) ds_1,$$
where the inner integral is a double Riemann-Stieltjes integral over the last two coordinates and the outer integral is a usual Riemann or Lebesgue integral. Furthermore, we get
\begin{equation*}
\begin{split}
&\int_0^{t_1} \int_0^{t_2} \int_0^{t_3}  d_{\{2,3\}} X(s) ds_1 = \int_0^{t_1} [X(s)]_{0_{\{2,3\}}}^{t_{\{2,3\}}} ds_1 \\
&=\int_0^{t_1} X(s_1,t_2,t_3) - X(s_1,t_2,0)- X(s_1,0,t_3) + X(s_1,0,0) ds_1.
\end{split}
\end{equation*}
\end{ex}

\begin{defi}
\label{defi:GH}
Let $\Theta \in (0,\infty)^N$ and $G= \{G(t)\}_{t\in\mathbb{R}^N}$ be a stationary increment field. If
\begin{equation*}
\lim_{s\to\infty} \int_{-s}^{t_1}\dots \int_{-s}^{t_N} e^{\Theta\cdot u} dG(u)
\end{equation*}
converges in probability and defines an almost surely finite random variable for all $t\in\mathbb{R}^N$, then we denote $G\in\mathcal{G}_\Theta$. In this case, the limit is denoted by $\int_{-\infty}^te^{\Theta\cdot u} dG(u)$. 

Furthermore, if $G(t) = 0$ for all $t$ such that $\sum_{l=1}^N t_l = 0$, then we write $G\in \mathcal{G}_{\Theta,0}$. We define an equivalence relation in $G\in \mathcal{G}_{\Theta,0}$ as follows. Let $G, \tilde{G}\in\mathcal{G}_{\Theta,0}.$ If $\tilde{G}(t) = G(t) + F(t)$, where 
\begin{equation*}
F(t) = \sum_{\substack{\emptyset\neq v\subseteq \{1,\dots,N\}\\
|v| \neq N}} f^{(v)}(t_v)
\end{equation*}
and $F(t) = 0$ for all $t$ such that $\sum_{l=1}^N t_l = 0$, then $\tilde{G} \sim G$, and we write $\tilde{G} = G$.
\end{defi}

\begin{rem}
\label{rem:equivalence}
Notice that $dF(t) = 0$, since every term in the defining sum depends only on a subset $t_v$ of all variables $t$. This together with $F(t)=0$ for all $t$ such that $\sum_l t_l =0$ ensures that $\tilde{G}(t) = G(t) + F(t)\in\mathcal{G}_{\Theta,0}.$ 
\end{rem}

\begin{rem}
\label{rem:existence}
Since $G$ is continuous and $e^{\Theta\cdot u}$ is, by Lemma \ref{lemma:variationofdifferentiable}, of bounded variation in the sense of Hardy and Krause, by Lemmas \ref{lemma:existence} and \ref{lemma:ipp}, the integral
$$\int_s^t e^{\Theta\cdot u} dG(u)$$
exists for all $t,s \in\mathbb{R}^N$.
\end{rem}

\section{On multiple Riemann-Stieltjes integrals}
\label{sec:RS}
In this section we present our results concerning multiple Riemann-Stieltjes integrals. We begin with the product rule and Radon-Nikodym type results in Section \ref{subsec:RS-standard}. The main difficulty in obtaining our characterizations, provided in Section \ref{sec:characterisation}, is to define Riemann-Stieltjes integrals over hypertriangles in a suitable way. This is the topic of Section \ref{subsec:RS-triangle}.
\subsection{On integration over  hyperrectangles}
\label{subsec:RS-standard}
\begin{theorem}
\label{lemma:radon}
Let $f$ be a function such that $f_t$ exists and is continuous on $[s,t]$. Then
$$d_u f(r) = f_{t_u}(r)dr_u\quad \text{for all } u\subseteq\{1,\dots,N\}$$
in the multiple Riemann-Stieltjes sense for every continuous integrand $g$. 
\begin{proof}
It suffices to consider the case $u = \{1,\dots,N\}$. The result for lower dimensional differentials follow by fixing variables $r_{-u}$ and considering $f$ as a $|u|$-variate function. Since the integrand $g$ is continuous, both integrals, Riemann-Stieltjes on the left and Riemann on the right, exist. Denote by
$$\left|I_{i_1,\dots,i_N}\right| = \prod_{j=1}^N (x_{i_j}^{(N)} -x_{i_j-1}^{(N)}) $$
the Lebesgue measure of $I_{i_1,\dots,i_N}$. We examine a summand of the Riemann-Stieltjes sum \eqref{rssum}. By the mean value theorem
\begin{equation*}
\begin{split}
 g(\xi_i)[f]_{(x^{(1)}_{i_1-1}, \dots, x^{(N)}_{i_N-1})}^{(x^{(1)}_{i_1}, \dots, x^{(N)}_{i_N})} =
 g(\xi_i) f_{t}(\eta_{i})\left|I_{i_1,\dots,i_N}\right|
 \end{split}
\end{equation*}
with some $\eta_{i}\in I_{i_1,\dots,i_N}$. 
Since $g$ is Riemann-Stieltjes integrable with respect to $f$, we may select $\xi_i = \eta_{i}$
yielding
\begin{equation*}
\begin{split}
&\lim_{\Vert P \Vert\to 0} \sum_{i_1 = 1}^{n_1}\dots\sum_{i_N = 1}^{n_N}g(\xi_i)[f]_{(x^{(1)}_{i_1-1}, \dots, x^{(N)}_{i_N-1})}^{(x^{(1)}_{i_1}, \dots, x^{(N)}_{i_N})}\\
&= \lim_{\Vert P \Vert\to 0}\sum_{i_1 = 1}^{n_1}\dots\sum_{i_N = 1}^{n_N} g(\eta_i) f_{t}(\eta_i)\left|I_{i_1,\dots,i_N}\right|,
\end{split}
\end{equation*}
where the right-hand side converges to the Riemann integral of $gf_t$.
\end{proof}
\end{theorem}

\begin{theorem}
\label{theorem:radon2}
Let $v\subseteq \{1,\dots,N\}$. Let $f$ be such that $f_{t_v}$ is continuous on $[s,t]$. Then
\begin{equation}
\label{proofthis}
d f(r) = d_{-v}f_{t_v}(r)dr_v
\end{equation}
in the multiple Riemann-Stieltjes sense for every integrand $g$ such that $g_t$ is continuous. 
\begin{proof}
The statement obviously holds for $v=\emptyset$. In addition, by Theorem \ref{lemma:radon}, it holds also for $v = \{1,\dots,N\}.$ Hence, we may exclude these cases in the proof. For other choices of $v$, the right-hand side of \eqref{proofthis} reads
\begin{equation}
\label{substitutehere}
\begin{split}
\int_s^t g(r) d_{-v} f_{t_v} (r) dr_v &= \int_{s_v}^{t_v} \int_{s_{-v}}^{t_{-v}} g(r) d_{-v} f_{t_v} (r) dr_v\\
&= \int_{s_v}^{t_v} \sum_{u\subseteq -v} (-1)^{|u|} \int_{s_u}^{t_u} \left[ f_{t_v}(r) d_u g(r)\right]_{s_{-v-u}}^{t_{-v-u}} dr_v\\
&= \int_{s_v}^{t_v} \sum_{u\subseteq -v} (-1)^{|u|} \int_{s_u}^{t_u} \left[ f_{t_v}(r) g_{t_u}(r) dr_u\right]_{s_{-v-u}}^{t_{-v-u}} dr_v \\
&= \sum_{u\subseteq -v} (-1)^{|u|} \int_{s_u}^{t_u}  \int_{s_v}^{t_v} \left[ f_{t_v}(r) g_{t_u}(r) dr_v\right]_{s_{-v-u}}^{t_{-v-u}} dr_u
\end{split}
\end{equation}
by integration by parts, Theorem \ref{lemma:radon} and Fubini's theorem. Next, we show via induction that
$$\int_{s_v}^{t_v} \left[ f_{t_v}(r) g_{t_u}(r) dr_v\right]_{s_{-v-u}}^{t_{-v-u}} = \sum_{w\subseteq v} (-1)^{|w|} \int_{s_w}^{t_w} \left[f(r) g_{t_{u+w}}(r)\right]_{s_{-u-w}}^{t_{-u-w}} dr_w. $$
For this, we assume, without loss of generality, that $v = \{1,\dots, |v|\}.$ In addition, we denote $v^{(n)} =\{1,\dots,|v|-n\}$ so that $v^{(0)} = v$. We pose the induction hypothesis
\begin{equation}
\label{inductionassumption}
\begin{split}
&\int_{s_v}^{t_v} \left[ f_{t_v}(r) g_{t_u}(r) dr_v\right]_{s_{-v-u}}^{t_{-v-u}}\\ 
&= \int_{s_{v^{(n)}}}^{t_{v^{(n)}}} \sum_{w\subseteq v- v^{(n)}} (-1)^{|w|} \int_{s_w}^{t_w}\left[f_{t_{v^{(n)}}}(r) g_{t_{u+w}}(r)\right]_{s_{-u-w-v^{(n)}}}^{t_{-u-w-v^{(n)}}}dr_w dr_{v^{(n)}},
\end{split}
\end{equation}
which a priori holds for $n=0$, giving also the base case. Next, we proceed to the induction step. We apply integration by parts coordinate by coordinate together with Fubini's theorem, starting from $r_{|v|-n}.$ The approach yields that \eqref{inductionassumption} is equal to
\begin{equation*}
\begin{split}
  &\int_{s_{v^{(n-1)}}}^{t_{v^{(n-1)}}} \sum_{w\subseteq v- v^{(n)}} (-1)^{|w|} \int_{s_w}^{t_w}\int_{s_{|v|-n}}^{t_{|v|-n}}\left[f_{t_{v^{(n)}}}(r) g_{t_{u+w}}(r)\right]_{s_{-u-w-v^{(n)}}}^{t_{-u-w-v^{(n)}}}dr_{|v|-n}dr_w dr_{v^{(n-1)}},
  \end{split}
\end{equation*}
from which integration by parts gives
\begin{equation*}
\begin{split}
&\sum_{w\subseteq v- v^{(n)}} (-1)^{|w|} \int_{s_w}^{t_w}\int_{s_{|v|-n}}^{t_{|v|-n}}\left[f_{t_{v^{(n)}}}(r) g_{t_{u+w}}(r)\right]_{s_{-u-w-v^{(n)}}}^{t_{-u-w-v^{(n)}}}dr_{|v|-n}dr_w\\
  & =\sum_{w\subseteq v- v^{(n)}} (-1)^{|w|} \int_{s_w}^{t_w}\left[f_{t_{{v^{(n-1)}}}}(r) g_{t_{u+w}}(r)\right]_{s_{-u-w-{v^{(n-1)}}}}^{t_{-u-w-{v^{(n-1)}}}}dr_w \\
  &\quad - \sum_{w\subseteq v- v^{(n)}} (-1)^{|w|} \int_{s_{w+\{|v|-n\}}}^{t_{w+\{|v|-n\}}}\left[f_{t_{{v^{(n-1)}}}}(r) g_{t_{u+w+\{|v|-n\}}}(r)\right]_{s_{-u-w-v^{(n)}}}^{t_{-u-w-v^{(n)}}}dr_{w+\{|v|-n\}} . 
  \end{split}
\end{equation*}
For the first term, we note that the summation is over the set $q\subseteq v - \{1,\dots,|v|-n-1\} = v-v^{(n-1)}$ with $|v| -n \notin q$. Moreover, by the change of variable $q=w+\{|v|-n\}$, implying $-q - \{1,\dots,|v|-n-1\}= -w-\{1,\dots,|v|-n\}$, the latter term can be written as
$$\sum_{\substack{q\subseteq v- v^{(n-1)}\\ |v|-n \in q}} (-1)^{|q|} \int_{s_q}^{t_q}\left[f_{t_{v^{(n-1)}}}(r) g_{t_{u+q}}(r)\right]_{s_{-u-q-v^{(n-1)}}}^{t_{-u-q-v^{(n-1)}}}dr_q .$$
Hence, \eqref{inductionassumption} is equal to
$$\int_{s_{v^{(n-1)}}}^{t_{v^{(n-1)}}} \sum_{\substack{q\subseteq v- v^{(n-1)}}} (-1)^{|q|} \int_{s_q}^{t_q}\left[f_{t_{v^{(n-1)}}}(r) g_{t_{u+q}}(r)\right]_{s_{-u-q-v^{(n-1)}}}^{t_{-u-q-v^{(n-1)}}}dr_q dr_{v^{(n-1)}} $$
completing the induction step. Furthermore, choosing $n = |v|$ in \eqref{inductionassumption} gives
$$\int_{s_v}^{t_v} \left[ f_{t_v}(r) g_{t_u}(r) dr_v\right]_{s_{-v-u}}^{t_{-v-u}}= \sum_{w\subseteq v} (-1)^{|w|} \int_{s_w}^{t_w} \left[f(r) g_{t_{u+w}}(r)\right]_{s_{-u-w}}^{t_{-u-w}} dr_w.$$
Finally, substitution to \eqref{substitutehere} gives
\begin{equation*}
\begin{split}
&\int_s^t g(r) d_{-v} f_{t_v}(r) dr_v\\ &= \sum_{u\subseteq -v} (-1)^{|u|} \int_{s_u}^{t_u} \sum_{w\subseteq v} (-1)^{|w|} \int_{s_w}^{t_w} \left[f(r) g_{t_{u+w}}(r)\right]_{s_{-u-w}}^{t_{-u-w}} dr_w dr_u\\
&= \sum_{u\subseteq -v}\sum_{w\subseteq v}(-1)^{|u+w|} \int_{s_{u+w}}^{t_{u+w}}\left[f(r) g_{t_{u+w}}(r)\right]_{s_{-u-w}}^{t_{-u-w}} dr_{u+w}\\
&= \sum_{q\subseteq \{1,\dots,N\}}(-1)^{|q|} \int_{s_{q}}^{t_{q}}\left[f(r) g_{t_{q}}(r)\right]_{s_{-q}}^{t_{-q}} dr_{q} = \int_s^t g(r) df(r),
\end{split}
\end{equation*}
where the last equality follows directly by integration by parts. This completes the proof.
\end{proof}
\end{theorem}

Next, we prove a product rule for multiple Riemann-Stieltjes differentials.
\begin{theorem}
\label{theorem:productrule}
Let $f$ be a function such that $f_t$ exists and is continuous on $[s,t]$ and let $g$ be continuous.
Then 
\begin{equation}
\label{productrule}
d(f(r)g(r)) = \sum_{u\subseteq \{1,\dots,N\}} f_{t_u}(r) d_{-u}g(r) dr_u
\end{equation}
in the multiple Riemann-Stieltjes sense for every integrand $h$ such that $h_t$ exists and is continuous on $[s,t]$.
\begin{proof}
By integration by parts and Theorem \ref{lemma:radon}, the left-hand side yields
\begin{equation*}
\begin{split}
&\int_s^t h(r) d(f(r)g(r)) = [h(r)f(r)g(r)]_s^t + \sum_{\emptyset\neq v \subseteq \{1,\dots,N\}}(-1)^{|v|} \int_{s_v}^{t_v} [g(r)f(r)d_vh(r)]_{s_{-v}}^{t_{-v}}\\
&= [h(r)f(r)g(r)]_s^t +\sum_{\emptyset\neq v\subseteq \{1,\dots,N\}}(-1)^{|v|} \int_{s_v}^{t_v} [g(r)f(r)h_{t_v}(r)dr_v]_{s_{-v}}^{t_{-v}}.
 \end{split}
\end{equation*}
We apply Theorem \ref{lemma:radon} and Lemma \ref{lemma:ipp} repeatedly in the rest of the proof.
The right-hand side of \eqref{productrule} gives
\begin{equation*}
\begin{split}
\int_s^t h(r)f(r) dg(r) +  \sum_{\emptyset\neq u\subseteq \{1,\dots,N\}}\int_s^t  h(r)f_{t_u}(r) d_{-u}g(r) dr_u,
\end{split}
\end{equation*}
where
\begin{equation*}
\int_s^t h(r)f(r) dg(r) = [h(r)f(r)g(r)]_s^t + \sum_{\emptyset\neq v \subseteq \{1,\dots,N\}}(-1)^{|v|} \int_{s_v}^{t_v} [g(r)d_v(h(r)f(r))]_{s_{-v}}^{t_{-v}}
\end{equation*}
and
\begin{equation*}
\begin{split}
&\sum_{\emptyset\neq u\subseteq \{1,\dots,N\}}\int_s^t  h(r)f_{t_u}(r) d_{-u}g(r) dr_u = \sum_{\emptyset\neq u\subseteq \{1,\dots,N\}}\int_{s_u}^{t_u} \int_{s_{-u}}^{t_{-u}} h(r)f_{t_u}(r) d_{-u}g(r) dr_u \\
&=\int_{s}^{t} h(r)f_t(r)g(r) dr + \sum_{\substack{u\subseteq\{1,\dots,N\}\\
1 \leq |u|\leq N-1}} \int_{s_u}^{t_u} [h(r)f_{t_u}(r)g(r)]_{s_{-u}}^{t_{-u}} dr_u\\
&\quad+ \sum_{\substack{u\subseteq\{1,\dots,N\}\\
1 \leq |u|\leq N-1}}\int_{s_u}^{t_u} \sum_{\substack{v\subseteq -u\\ 1\leq |v|}} (-1)^{|v|} \int_{s_v}^{t_v} [g(r) d_v(h(r)f_{t_u}(r))]_{s_{-u-v}}^{t_{-u-v}} dr_u.
\end{split}
\end{equation*}
Note that the integrals in the last term are well-defined by Lemma \ref{lemma:variationofdifferentiable}, since the partial derivative of $hf_{t_u}$ with respect to $t_v$ is continuous. By combining the previous two equations, the right-hand side of \eqref{productrule} gives
\begin{equation*}
\begin{split}
&[h(r)f(r)g(r)]_s^t +\sum_{\emptyset\neq v \subseteq \{1,\dots,N\}}(-1)^{|v|} \int_{s_v}^{t_v} [g(r)\frac{dhf}{dt_v}(r) dr_v]_{s_{-v}}^{t_{-v}}\\
&\quad+\int_{s}^{t} h(r)f_t(r)g(r) dr + \sum_{\substack{u\subseteq\{1,\dots,N\}\\
1 \leq |u|\leq N-1}} \int_{s_u}^{t_u} [h(r)f_{t_u}(r)g(r)dr_u]_{s_{-u}}^{t_{-u}}\\
&\quad+ \sum_{\substack{u\subseteq\{1,\dots,N\}\\
1 \leq |u|\leq N-1}} \sum_{\substack{v\subseteq -u\\ 1\leq |v|}} (-1)^{|v|} \int_{s_{u+v}}^{t_{u+v}} \left[g(r) \frac{dhf_{t_u}}{dt_v}(r) dr_{u+v}\right]_{s_{-u-v}}^{t_{-u-v}}.
\end{split}
\end{equation*}
That is, in order to show that Equation \eqref{productrule} holds in the given sense, we need to show that
\begin{equation*}
\begin{split}
&\sum_{\emptyset\neq v \subseteq \{1,\dots,N\}}(-1)^{|v|} \int_{s_v}^{t_v} [g(r)\sum_{\substack{u\subseteq v\\ u \neq v}}h_{t_u}(r)f_{t_{v-u}}(r) dr_v]_{s_{-v}}^{t_{-v}}\\
&\quad +\int_{s}^{t} h(r)f_t(r)g(r) dr + \sum_{\substack{u\subseteq\{1,\dots,N\}\\
1 \leq |u|\leq N-1}} \int_{s_u}^{t_u} [h(r)f_{t_u}(r)g(r)dr_u]_{s_{-u}}^{t_{-u}}\\
&\quad+ \sum_{\substack{u\subseteq\{1,\dots,N\}\\
1 \leq |u|\leq N-1}} \sum_{\substack{v\subseteq -u\\ 1\leq |v|}} (-1)^{|v|} \int_{s_{u+v}}^{t_{u+v}} \left[g(r) \sum_{w\subseteq v} h_{t_w}(r) f_{t_{u+(v-w)}}(r) dr_{u+v}\right]_{s_{-u-v}}^{t_{-u-v}}\\
&= \sum_{\emptyset\neq v \subseteq \{1,\dots,N\}}(1+(-1)^{|v|}) \int_{s_v}^{t_v} [g(r)h(r)f_v(r) dr_v]_{s_{-v}}^{t_{-v}}\\
&\quad+ \sum_{\substack{u\subseteq\{1,\dots,N\}\\
1 \leq |u|\leq N-1}} \sum_{\substack{v\subseteq -u\\ 1\leq |v|}} (-1)^{|v|} \int_{s_{u+v}}^{t_{u+v}} \left[g(r) h(r) f_{t_{u+v}}(r) dr_{u+v}\right]_{s_{-u-v}}^{t_{-u-v}}\\
&\quad+ \sum_{\substack{v\subseteq \{1,\dots,N\} \\ 2 \leq |v|}}(-1)^{|v|} \int_{s_v}^{t_v} [g(r)\sum_{\substack{u\subseteq v\\ 1\leq |u| \leq |v|-1}}h_{t_u}(r)f_{t_{v-u}}(r) dr_v]_{s_{-v}}^{t_{-v}}\\
&\quad+ \sum_{\substack{u\subseteq\{1,\dots,N\}\\
1 \leq |u|\leq N-1}} \sum_{\substack{v\subseteq -u\\ 1\leq |v|}} (-1)^{|v|} \int_{s_{u+v}}^{t_{u+v}} \left[g(r) \sum_{\substack{w\subseteq v\\1\leq |w|}} h_{t_w}(r) f_{t_{u+(v-w)}}(r) dr_{u+v}\right]_{s_{-u-v}}^{t_{-u-v}} =0.
\end{split}
\end{equation*}
The first term after the equality follows by choosing $u = \emptyset$ in the first term of the equation and combining it with the following two terms. The third term after the equality is the first term of the equation with the additional restriction $u \neq \emptyset$. 
We show that the sum of the last two terms equals zero. The first two terms can be handled in a similar, but less involved manner. Thereby, it suffices to show
\begin{equation}
\label{zero}
\begin{split}
&\sum_{\substack{q\subseteq \{1,\dots,N\} \\ 2 \leq |q|}}(-1)^{|q|} \int_{s_q}^{t_q} [g(r)\sum_{\substack{w\subseteq q\\ 1\leq |w| \leq |q|-1}}h_{t_w}(r)f_{t_{q-w}}(r) dr_q]_{s_{-q}}^{t_{-q}}\\
&\quad+ \sum_{\substack{u\subseteq\{1,\dots,N\}\\
1 \leq |u|\leq N-1}} \sum_{\substack{v\subseteq -u\\ 1\leq |v|}} (-1)^{|v|} \int_{s_{u+v}}^{t_{u+v}} \left[g(r) \sum_{\substack{\tilde{w}\subseteq v\\1\leq |\tilde{w}|}} h_{t_{\tilde{w}}}(r) f_{t_{u+(v-\tilde{w})}}(r) dr_{u+v}\right]_{s_{-u-v}}^{t_{-u-v}} = 0.
\end{split}
\end{equation}
Now, let $q \subseteq \{1,\dots,N\}$ with $|q|\geq 2$ be fixed. Also, let $w\subseteq q$ with $1\leq |w|\leq |q|-1$ be fixed. By doing this, we have fixed a unique summand of the first term of \eqref{zero}. The sign of this term is positive if $|q|$ is even and negative if $|q|$ is odd. We next consider which terms in the latter double sum of \eqref{zero} correspond to the choices $q$ and $w$. First of all, it has to hold that $w\subseteq v$, but $v$ can also contain indices of $q-w$. In addition, $u+v=q$ has to hold. That is, by fixing $v$ we also fix $u$. We now have four different scenarios, where $|q|$ and $|w|$ can be even or odd. We start by assuming that $|q|$ and $|w|$ are both even. By Lemma \ref{binomial}, the sum related to the corresponding terms in \eqref{zero} equals
\begin{equation}
\label{zero1}
\binom{|q|-|w|}{0} - \binom{|q|-|w|}{1}+\dots -\binom{|q|-|w|}{|q|-|w|-1} + 1 = 0,
\end{equation}
where the binomial coefficient $\binom{|q|-|w|}{i}$ corresponds to $v$ which contains $w$ and $i$ elements from $q-w$. Here we have used the fact that since $u+v = q$ and $|u|,|v| \geq 1$, $i$ runs from 0 to $|q| - |w| -1$. By considering in a similar way the other cases, we see that regardless of whether $|q|$ and $|w|$ are even or odd, the sum equals zero by Lemma \ref{binomial}. 
\end{proof}

\end{theorem}

\begin{lemma}
\label{lemma:fundamental2}
    Let $f$ be continuous and $F_v(t) = \int_{s_v}^{t_v} [f(r)] _{s_{-v}}^{t_{-v}}dr_v$. Then 
    $$dF_v(t) = d\int_{s_v}^{t_v} f(r_v:t_{-v}) dr_v = d_{-v}(f(t))dt_v$$
    for every integrand $g$ such that $g_t$ is continuous.
    \begin{proof}
    Notice first that
    $$\int_{s_v}^{t_v} [f(r)] _{s_{-v}}^{t_{-v}}dr_v =\int_{s_v}^{t_v} \sum_{w\subseteq -v} (-1)^{|w|} f(r_v:s_w:t_{-v-w}) dr_v $$ 
    depends on all $N$ coordinates only for $w=\emptyset$. In all other cases, the square increments are identically zero (see also Theorem 2.2.5 in \cite{prause2015sequential}) and thus, we obtain
    $$d\int_{s_v}^{t_v} \sum_{\substack{{w\subseteq -v}\\ w\neq\emptyset}} (-1)^{|w|} f(r_v:s_w:t_{-v-w}) dr_v=0.$$
    We have established that $$dF_v(t) = d\int_{s_v}^{t_v} f(r_v:t_{-v}) dr_v$$
    and Theorem \ref{theorem:radon2} completes the proof.

 \end{proof}
\end{lemma}

\subsection{Integration over non-hyperrectangles}
\label{subsec:RS-triangle}
Next, we prepare to discuss integrals over non-standard domains. For $\sum_{l=1}^N t_l \geq 0$ we set 
\begin{equation}
\label{T}
T = T(t) = \{x\in\mathbb{R}^N : x_l \leq t_l\text{ for all } l, \sum x \geq 0\},
\end{equation}
where $\sum x$ is the sum of all elements of the vector $x$.
It can be shown that
\begin{equation*}
\begin{split}
T = \left\{x\in\mathbb{R}^N : -\sum_{l\neq 1}t_l \leq x_1 \leq t_1, -x_1-\sum_{l \notin \{1,2\}} t_l \leq x_2 \leq t_2,\right.\\
\left.\dots, -\sum_{l\neq N}x_l \leq x_N \leq t_N \right\}.
\end{split}
\end{equation*}
Similarly, if $v \subseteq \{1,\dots,N\}$, then
$$T_v = T_v(t)= \{ x\in\mathbb{R}^{|v|} : x_l \leq t_{v_l} \text{ for all } l, \sum x+ \sum t_{-v} \geq 0\},$$
where $T_{\{1,\dots,N\}} = T.$
It can be shown that
\begin{equation}
\label{T_v}
\begin{split}
T_v = \left\{-\sum_{l\neq v_1} t_l \leq x_1 \leq t_{v_1}, -x_1-\sum_{l \notin \{v_1,v_2\}} t_l \leq x_2 \leq t_{v_2},\right. \\
\left. \dots, -\sum_{l\neq |v|}x_l  - \sum t_{-v}\leq x_{|v|}\leq t_{v_{|v|}} \right\}.
\end{split}
\end{equation}
Moreover, we set
$$\tilde{t} = \left( -\sum_{l\neq 1} t_l, \dots, -\sum_{l \neq N}t_l\right)$$
and hence
$$\tilde{t}_v = \left( -\sum_{l\neq v_1} t_l, \dots, -\sum_{l \neq v_{|v|}}t_l\right).$$
Now $ R \coloneqq [\tilde{t},t]$ is a hypercube containing $T$ and $R_v \coloneqq [\tilde{t}_v,t_v]$ is a $|v|$-dimensional hypercube containing $T_v$. Notice that $\tilde{t}_l = t_l$ for some $l$ if and only if $\sum_l t_l = 0$. Moreover, then $\tilde{t}_j = -\sum_l t_l = t_j$ for all $j$. That is, if $\sum_l t_l = 0$, the rectangle $R$ (and $R_v$) is a single point. 

\begin{ex}
\label{ex:T}
If $N=2$, then $R$ is a square and the subcubes $R_v$ correspond to the sides of $R$, where the coordinate $-v$ is fixed to $t_{-v}$ or $\tilde{t}_{-v}$. Moreover, the line $t_1+t_2 =0$ divides $R$ into two triangles along the corner points $(t_1, \tilde{t}_2)$ and $(\tilde{t}_1, t_2)$. The upper right triangle is $T$, and $T_v$ correspond to the catheti of $T$.  If $N=3$, then $R$ is a cube and the subcubes $R_v$ correspond to the faces and edges of $R$. Moreover, the plane $t_1+t_2+t_3 = 0$ cuts $R$ into two pieces along the corner points $(t_1,t_2,\tilde{t}_3), (t_1, \tilde{t}_2,t_3)$ and $(\tilde{t}_1, t_2,t_3)$ that are neighbours of $t$. The resulting tetrahedron is $T$, and $T_v$ correspond to the faces and edges of $T$ that are not subsets of the cutting plane.
\end{ex}
Our aim is to define integrals of the type
\begin{equation}
\label{integral}
 \int_T f dg,
\end{equation}
where $f$ is such that $f_t$ is continuous on $R$ and $g$ is continuous on $R$. We require the natural condition that our integrals satisfy 
$$\int_R f dg = \int_T f dg + \int_{R\setminus T} f dg,$$
where now
\begin{equation}
\label{wholeintegral}
\begin{split}
&\int_R f dg = [fg]_{\tilde{t}}^t + \sum_{\emptyset \neq v\{1,\dots,N\}} (-1)^{|v|} \int_{R_v} [f_{t_v}(u)g(u) du_v]_{\tilde{t}_{-v}}^{t_{-v}}
= [fg]_{\tilde{t}}^t\\
&\quad + \sum_{\emptyset \neq v \{1,\dots,N\}} (-1)^{|v|} \sum_{w\subseteq -v} (-1)^{|w|}\int_{R_v} f_{t_v}(t_{-v-w}:\tilde{t}_w:u_v)g(t_{-v-w}:\tilde{t}_w:u_v) du_v
\end{split}
\end{equation}
by integration by parts and Theorem \ref{lemma:radon}. For $v\neq \emptyset, w\subseteq -v$, the integral $\int_{R_v}$ is associated with the set
\begin{equation}
\label{set2}
\{x\in\mathbb{R}^N: x_v \in R_v, x_w = \tilde{t}_w, x_{-v-w} = t_{-v-w}\} \eqqcolon \mathcal{R}_{v,w},
\end{equation}
where the integration is over the coordinates $x_v$, while the other coordinates are appropriately fixed. 

\begin{lemma}
\label{lemma:intersection}
For sets \eqref{T} and \eqref{set2}, we have that
\begin{enumerate}[label=(\roman*)]
\item When $w = \emptyset$, then
$$T \cap \mathcal{R}_{v, \emptyset} = \{x\in \mathbb{R}^N : x_v\in T_v, x_{-v} = t_{-v}\}, $$
which is associated with the integral $\int_{T_v}$.
\item When $w \neq \emptyset$, then
$$\mu_{|v|}(T \cap \mathcal{R}_{v, w}) = 0,$$
where $\mu_{|v|}$ is the $|v|$-dimensional Lebesgue measure.
\end{enumerate}
\begin{proof}
(i) 
\begin{equation*}
\begin{split}
T \cap \mathcal{R}_{v, \emptyset} &= \{x\in\mathbb{R}^N: x_v \in [\tilde{t}_v, t_v], x_{-v} = t_{-v}\} \cap \{x\in\mathbb{R}^N: x \leq t, \sum x \geq 0\}\\
&= \{x\in\mathbb{R}^N : x_{-v} = t_{-v}, x_v\in [\tilde{t}_v, t_v], \sum x \geq 0\}\\
&= \{x\in\mathbb{R}^N : x_{-v} = t_{-v}, x_v\in [\tilde{t}_v, t_v], \sum t_{-v} + \sum x_v \geq 0\}\\
&= \{x\in\mathbb{R}^N : x_{-v} = t_{-v}, x_v\leq t_v, \sum t_{-v} + \sum x_v \geq 0\}\\
&= \{x\in \mathbb{R}^N : x_v\in T_v, x_{-v} = t_{-v}\},
\end{split}   
\end{equation*}
where the penultimate equation follows, since, for example
$$x_{v_1} \geq - \sum t_{-v} - \sum_{l\neq 1} x_{v_l} \geq - \sum t_{-v} - \sum_{l\neq 1} t_{v_l} = \tilde{t}_{v_1}.$$
(ii) Now
$$T \cap \mathcal{R}_{v, w} =  \{x\in\mathbb{R}^N : x_w = \tilde{t}_w, x_{-v-w} = t_{-v-w}, x_v\in [\tilde{t}_v, t_v], \sum x \geq 0\}.$$
Assume that $|w| = 1$. Then
$$ \sum x = \tilde{t}_{w_1} + \sum_{l\neq w_1} x_l \leq -\sum_{l \neq w_1} t_l + \sum_{l\neq w_1} t_l = 0.$$
That is, $T \cap \mathcal{R}_{v, w}$ is a single point set
$$\{ x\in\mathbb{R}^N : x_w = \tilde{t}_w : x_{-w} = t_{-w}\}.$$
In a similar way, if $|w| \geq 2$, then $T \cap \mathcal{R}_{v, w} = \emptyset$ unless we are in the degenerate case $\sum t = 0$ obtaining a set of one point. 
\end{proof}
\end{lemma}
Based on Lemma \ref{lemma:intersection}, we allocate terms
$$\sum_{\emptyset\neq v \subseteq \{1,\dots,N\}} (-1)^{|v|}\int_{T_v} f_{t_v}(t_{-v}:u_v) g(t_{-v}:u_v)du_v$$
of \eqref{wholeintegral} to $\int_T f dg$ and terms 
\begin{equation*}
\begin{split}
& \sum_{\emptyset\neq v \subseteq \{1,\dots,N\}} (-1)^{|v|}\left(\int_{R_v\setminus T_v}  f_{t_v}(t_{-v}:u_v) g(t_{-v}:u_v)du_v \right.\\
&\quad+ \left. \sum_{\emptyset \neq w \subseteq -v} (-1)^{|w|} \int_{R_v}  f_{t_v}(t_{-v-w}: \tilde{t}_w:u_v) g(t_{-v-w}: \tilde{t}_w:u_v) du_v \right)
\end{split}
\end{equation*}
to $\int_{R\setminus T} f dg.$ Next, we consider the allocation of $[fg]_{\tilde{t}}^t$ present in \eqref{wholeintegral}. Since $t\in T$ and $t\notin \overline{R\setminus T},$ we put $f(t)g(t)$ into $\int_T f dg$. Moreover, similarly to earlier, it can be verified that $(t_{-v}: \tilde{t}_v) \in R\setminus T$, whenever $|v| \geq 2.$ Thus, we allocate such terms $f(t_{-v}: \tilde{t}_{v})g(t_{-v}: \tilde{t}_{v})$ into $\int_{R\setminus T} fdg.$ When $|v| =1$, then $(t_{-v}: \tilde{t}_v) \in T \cap \overline{R\setminus T}$, so the allocation of $\sum_{|v|=1} f(t_{-v}:\tilde{t}_v)g(t_{-v}:\tilde{t}_v)$ is not entirely unambiguous. However, the points $(t_{-v}: \tilde{t}_v)$ are positioned symmetrically in our configuration and hence, we require equal weights for the corresponding terms. In addition, we want that $\int_T f dg =0$ whenever $\sum t = 0$. In this case, all integrals of type $\int_{T_v}$ vanish, since $T_v$ is a set of a single point. Moreover, recall that now $\tilde{t} = t$. Hence, for $\sum t =0$, we have
$$\int_T f dg = f(t)g(t) - a\sum_{|v|=1} f(t_{-v}:\tilde{t}_v)g(t_{-v}:\tilde{t}_v) = (1-aN) f(t)g(t) = 0$$
giving $a= \frac{1}{N}.$

Based on the above considerations, we define
\begin{equation}
\begin{split}
\label{T>0}
&\int_T f(u) dg(u) 
= f(t)g(t) - \frac{1}{N} \sum_{|v|=1} f(t_{-v}:\tilde{t}_v)g(t_{-v}:\tilde{t}_v)\\ &\quad + \sum_{\emptyset\neq v \subseteq \{1,\dots,N\}} (-1)^{|v|}\int_{T_v} f_{t_v}(t_{-v}:u_v) g(t_{-v}:u_v)du_v
\end{split}
\end{equation}
and 
\begin{equation*}
\begin{split}
&\int_{R \setminus T} f(u) dg(u) = - \frac{N-1}{N} \sum_{|v|=1} f(t_{-v}:\tilde{t}_v)g(t_{-v}:\tilde{t}_v)\\
&\quad + \sum_{|v| \geq 2} (-1)^{|v|} f(t_{-v}:\tilde{t}_v)g(t_{-v}:\tilde{t}_v)\\
&\quad + \sum_{\emptyset\neq v \subseteq \{1,\dots,N\}} (-1)^{|v|}\left(\int_{R_v\setminus T_v}  f_{t_v}(t_{-v}:u_v) g(t_{-v}:u_v)du_v \right.\\
&\quad + \left. \int_{R_v} \sum_{\emptyset \neq w \subseteq -v} (-1)^{|w|} f_{t_v}(t_{-v-w}: \tilde{t}_w:u_v) g(t_{-v-w}: \tilde{t}_w:u_v) du_v \right).
\end{split}
\end{equation*}
Indeed, by denoting $fg = h$ and $f_{t_v}g = h^{(v)}$, we have 
\begin{equation*}
\begin{split}
&\int_{R \setminus T} f(u) dg(u) + \int_{ T} f(u) dg(u) = h(t) - \sum_{|v|=1} h(t_{-v}:\tilde{t}_v) + \sum_{|v|\geq 2} (-1)^{|v|}h(t_{-v}:\tilde{t}_v)\\
&\quad + \sum_{\emptyset\neq v \subseteq \{1,\dots,N\}} (-1)^{|v|}\left( \int_{T_v} h^{(v)}(t_{-v}:u_v) du_v + \int_{R_v\setminus T_v}  h^{(v)}(t_{-v}:u_v)du_v\right. \\
&\quad + \left. \int_{R_v} \sum_{\emptyset \neq w \subseteq -v} (-1)^{|w|} h^{(v)}(t_{-v-w}: \tilde{t}_w:u_v)  du_v\right)\\
&= \sum_{v \subseteq \{1,\dots,N\}} (-1)^{|v|} h(t_{-v}:\tilde{t}_v)\\
&\quad + \sum_{\emptyset\neq v \subseteq \{1,\dots,N\}} (-1)^{|v|}\int_{R_v} \sum_{w\subseteq -v} (-1)^{|w|} h^{(v)}(t_{-v-w}: \tilde{t}_w:u_v)\\
&= [fg]_{\tilde{t}}^t + \sum_{\emptyset\neq v \subseteq \{1,\dots,N\}} (-1)^{|v|} \int_{R_v} \left[ h^{(v)}(u) du_v\right]_{\tilde{t}_{-v}}^{t_{-v}}\\
&= [fg]_{\tilde{t}}^t +  \sum_{\emptyset\neq v \subseteq \{1,\dots,N\}} (-1)^{|v|} \int_{\tilde{t}_v}^{t_v}  \left[f_{t_v}(u) g(u) du_v \right]_{\tilde{t}_{-v}}^{t_{-v}}\\ 
&= \int_{\tilde{t}}^t f(u) dg(u)= \int_R f(u) dg(u)
\end{split}
\end{equation*}
as desired. Note also that when $\sum t = 0$, then all the related domains are degenerate, and $\tilde{t} = t$. Since in this case $\int_R fdg =0$ and $\int_T fdg =0$ by the definitions, then also $\int_{R\setminus T} fdg =0$ by the additivity property shown above. 
\begin{rem}
Note that in our definitions, we have not considered terms that involve integration over the hyperplane that cuts the hypercube $R$ into two pieces, which might seem cumbersome at first sight. Indeed, in principle, one should also take into account the values on the cutting hyperplane by adding integrals defined in a natural manner. Note that the contribution of these integrals would be a sum of functions where each function depends only on a single corner point on the cutting hyperplane. By the construction of $T=T(t)$ (see also Example \ref{ex:T}), each of these corner points does not depend on all $N$ variables, and consequently, the integrals related to the cutting hyperplane have zero rectangular increments. In particular, for our purposes, the integration over the region $T$ is only required to define $G$ via such integrals in Lemma \ref{lemma:G}, which in turn is needed to obtain $dG(t) = e^{-\Theta\cdot t}dY(e^t)$ in the multiple Riemann-Stieltjes sense. As such, adding terms related to the cutting hyperplane would not change $dG(t)$ in the multiple Riemann-Stieltjes sense, and it is convenient to omit these terms, making our already notationally heavy computations simpler. 
\end{rem}

Next, we consider the setting with $\sum_{l=1}^N t_l  <0$. Now, we set
\begin{equation}
\label{T2}
T = T(t)= \{x\in\mathbb{R}^N : x_l \geq t_l\text{ for all } l, \sum x \leq 0\}.
\end{equation}
It can be shown that
\begin{equation*}
\begin{split}
T = \left\{x\in\mathbb{R}^N :  t_1 \leq x_1 \leq-\sum_{l\neq 1}t_l, t_2 \leq x_2 \leq-x_1-\sum_{l \notin \{1,2\}} t_l,\right.\\
\left.\dots, t_N \leq x_N \leq  -\sum_{l\neq N}x_l  \right\}.
\end{split}
\end{equation*}
Similarly, if $v \subseteq \{1,\dots,N\}$, then
$$T_v =T_v(t)= \{ x\in\mathbb{R}^{|v|} : x_l \geq t_{v_l} \text{ for all } l, \sum x+ \sum t_{-v} \leq 0\}.$$
It can be shown that
\begin{equation}
\begin{split}
\label{Tv2}
T_v = \left\{ t_{v_1}\leq x_1 \leq -\sum_{l\neq v_1} t_l,  t_{v_2} \leq x_2 \leq -x_1-\sum_{l \notin \{v_1,v_2\}} t_l,\right.\\
\left. \dots,  t_{v_{|v|}} \leq x_{|v|}\leq  -\sum_{l\neq |v|}x_l  - \sum t_{-v} \right\}.
\end{split}
\end{equation}
Now $R_v = [t_v,\tilde{t}_v]$ is a hyperrectangle containing $T_v$. Again, we take additivity as our starting point. By Remark \ref{rem:extension},
\begin{equation*}
\begin{split}
\int_R f dg  &= [fg]_t^{\tilde{t}} + \sum_{\emptyset \neq v \subseteq \{1,\dots, N\}} (-1)^{|v|} \int_{R_v} [f_{t_v}(u) g(u) du_v]_{t_{-v}}^{\tilde{t}_{-v}}\\
&= (-1)^N[fg]_{\tilde{t}}^t + \sum_{\emptyset \neq v \subseteq \{1,\dots, N\}} (-1)^{|v|} (-1)^{|-v|}\int_{R_v} [f_{t_v}(u) g(u) du_v]_{\tilde{t}_{-v}}^{t_{-v}}\\
&= (-1)^N\left([fg]_{\tilde{t}}^t + \sum_{\emptyset \neq v \subseteq \{1,\dots, N\}} \int_{R_v} [f_{t_v}(u) g(u) du_v]_{\tilde{t}_{-v}}^{t_{-v}} \right).
\end{split}
\end{equation*}
Similarly as in the case of $\sum t \geq 0$, this justifies defining 
\begin{equation}
\label{superintegral2}
\begin{split}
&\int_T f(u) dg(u)  = (-1)^N \left[f(t) g(t)  - \frac{1}{N} \sum_{|v| = 1} f(t_{-v} :\tilde{t}_v) g(t_{-v} :\tilde{t}_v)  \right.\\
&\quad + \left.
\sum_{\emptyset\neq v \subseteq \{1,\dots,N\}}\int_{T_v} {f}_{t_v}({t_{-v}:u_v}) g(t_{-v}:u_v) du_v\right]
\end{split}
\end{equation}
and

\begin{equation*}
\begin{split}
&\int_{R \setminus T} f(u) dg(u) = (-1)^N \Bigg[- \frac{N-1}{N} \sum_{|v|=1} f(t_{-v}:\tilde{t}_v)g(t_{-v}:\tilde{t}_v)\\
&\quad + \sum_{|v| \geq 2} (-1)^{|v|} f(t_{-v}:\tilde{t}_v)g(t_{-v}:\tilde{t}_v)\\
&\quad + \sum_{\emptyset\neq v \subseteq \{1,\dots,N\}} \left(\int_{R_v\setminus T_v}  f_{t_v}(t_{-v}:u_v) g(t_{-v}:u_v)du_v \right.\\
&\quad + \left. \int_{R_v} \sum_{\emptyset \neq w \subseteq -v} (-1)^{|w|} f_{t_v}(t_{-v-w}: \tilde{t}_w:u_v) g(t_{-v-w}: \tilde{t}_w:u_v) du_v \right) \Bigg].
\end{split}
\end{equation*}
We also note that, by the continuity of the involved functions and the fact that $\lim_{\sum t \to 0} \tilde{t} \to t$, we obtain directly from the definitions that 
$$\lim_{\sum t \to 0} \int_T f(u) dg(u) = 0 = \lim_{\sum t \to 0} \int_R f(u) dg(u) $$
implying also that $\lim_{\sum t \to 0} \int_{R\setminus T} f(u) dg(u) \to 0$. Therefore, our integrals are continuous with respect to $T = T(t)$ also on the hyperplane $\sum t = 0$.

\section{Characterization of continuous stationary fields}
\label{sec:characterisation}
In this section we provide our characterization of stationary fields via generalized Langevin equation. We begin with the following lemma.
\begin{lemma}
\label{lemma:Y}
Let $\Theta\in (0,\infty)^N$ and $G = \{G(t)\}_{t\in\mathbb{R}^N}\in \mathcal{G}_\Theta.$ Then 
\begin{equation*}
Y(e^t) = \int_{-\infty}^t e^{\Theta\cdot u} dG(u)
\end{equation*} 
is $\Theta$-self-similar. Moreover, $dY(e^t) = e^{\Theta\cdot t}dG(t)$ in the multiple Riemann-Stieltjes sense for every integrand $g$ such that $g_t$ is continuous.
\begin{proof}
If $G\in\mathcal{G}_\Theta$, then
$$ Y(e^{t+s}) = \int_{-\infty}^{t+s} e^{\Theta\cdot u} dG(u), $$
where the change of variable $r=u-s$ gives
$$Y(e^{t+s}) = \int_{-\infty}^{t} e^{\Theta\cdot (r+s) } dG(r+s) \overset{\text{law}}{=} e^{\Theta\cdot s}\int_{-\infty}^{t} e^{\Theta\cdot r } dG(r) =e^{\Theta\cdot s} Y(e^t) $$
by stationarity of increments of $G$. The equality of laws can be verified by studying Riemann-Stieltjes sums and passing to the limit. Treating multidimensional distributions similarly shows that $Y$ is $\Theta$-self-similar. 

For the latter claim, we consider integrals over a compact hyperrectangle, where the integrand $g$ has the continuous mixed partial derivative $g_t$. Let
$$Y_u(e^t) = \int_{-\bar{u}}^t e^{\Theta\cdot r} dG(r),$$
where $\bar{u} = (u,\dots, u)\in\mathbb{R}^N$ is such that $-u < t_i$ for all $i$.
Recall that if $R, R_1$, and $R_2$ are hyperrectangles such that $R = R_1\cup R_2$ and the interiors of $R_1$ and $R_2$ are disjoint, then
$$\int_R f dg = \int_{R_1} f dg + \int_{R_2} f dg$$
whenever the integrals exist. Hence, for large enough $u$ we have that
\begin{equation*}
\begin{split}
&\Delta_1\dots \Delta_N Y_u(e^{x_{i_1}^{(1)}}, \dots, e^{x_{i_N}^{(N)}})\\
 &=\Delta_1\dots \Delta_N \int_{-u}^{x_{i_1}^{(1)}}\dots \int_{-u}^{x_{i_N}^{(N)}}e^{\Theta\cdot r} dG(r)\\
 &= \Delta_1\dots \Delta_{N-1}\left(\int_{-u}^{x_{i_1}^{(1)}}\dots \int_{-u}^{x_{i_N}^{(N)}}e^{\Theta\cdot r} dG(r) - \int_{-u}^{x_{i_1}^{(1)}}\dots \int_{-u}^{x_{i_N-1}^{(N)}}e^{\Theta\cdot r} dG(r) \right )\\
 &= \Delta_1\dots \Delta_{N-1} \int_{-u}^{x_{i_1}^{(1)}}\dots \int_{-u}^{x_{i_{N-1}}^{(N-1)}}\int_{x_{i_N-1}^{(N)}}^{x_{i_N}^{(N)}}e^{\Theta\cdot r} dG(r)\\
 &= \dots = \int_{x_{i_1-1}^{(1)}}^{x_{i_1}^{(1)}} \dots \int_{x_{i_N-1}^{(N)}}^{x_{i_N}^{(N)}} e^{\Theta\cdot r} dG(r) = \Delta_1\dots \Delta_N Y(e^{x_{i_1}^{(1)}}, \dots, e^{x_{i_N}^{(N)}}).
\end{split}
\end{equation*}
Together with the compactness of the integration domain, this allows us to write $dY_u(e^t) = dY(e^t)$. By integration by parts, this yields
\begin{equation*}
\begin{split}
Y_u(e^t) &= \int_{-\bar{u}}^t e^{\Theta\cdot r} dG(r) = [e^{\Theta\cdot r} G(r)]_{-\bar{u}}^t + \sum_{\emptyset\neq v \subseteq \{1,\dots,N\}}(-1)^{|v|} \int_{-\bar{u}_v}^{t_v} [G(r)d_ve^{\Theta\cdot r}]_{-\bar{u}_{-v}}^{t_{-v}}\\
&= [e^{\Theta\cdot r} G(r)]_{-\bar{u}}^t + \sum_{\emptyset\neq v \subseteq \{1,\dots,N\}}(-1)^{|v|} \int_{-\bar{u}_v}^{t_v} [\prod_{j\in v} (\theta_j) G(r)e^{\Theta\cdot r} dr_v]_{-\bar{u}_{-v}}^{t_{-v}},
\end{split}
\end{equation*}
where
\begin{equation}
\label{combo}
\begin{split}
d [e^{\Theta\cdot r} G(r)]_{-\bar{u}}^t &= d(e^{\Theta \cdot t} G(t)) = \sum_{u\subseteq \{1,\dots,N\}} \prod_{j\in u} (\theta_j) e^{\Theta \cdot t} d_{-u}G(t) dt_u\\
&= e^{\Theta \cdot t} dG(t) + e^{\Theta \cdot t}\sum_{\substack{u\subseteq \{1,\dots,N\}\\ |u| \geq 1} } \prod_{j\in u} (\theta_j)  d_{-u}G(t) dt_u.
\end{split}
\end{equation}
The first two equalities follow from the fact that $u$ is a constant as the total differential is taken with respect to $t$ together with Theorem \ref{theorem:productrule} (see also Theorem 2.2.5 in \cite{prause2015sequential}). In addition, by Lemma \ref{lemma:fundamental2},

\begin{equation}
\label{combo2}
\begin{split}
d  \int_{-\bar{u}_v}^{t_v} [ G(r)e^{\Theta\cdot r} dr_v]_{-\bar{u}_{-v}}^{t_{-v}} &= 
 d_{-v}\left( G(t) e^{\Theta\cdot t} \right)dt_v\\
&= e^{\Theta \cdot t} \sum_{u\subseteq -v} \prod_{j\in u} (\theta_j) d_{-v-u} G(t) dt_u dt_v.
\end{split}
\end{equation}
By combining \eqref{combo} and \eqref{combo2}, we obtain that
\begin{equation}
\label{kukkuukukkuu}
\begin{split}
dY_u(e^t) &= e^{\Theta \cdot t} dG(t) + e^{\Theta \cdot t}\sum_{\substack{u\subseteq \{1,\dots,N\}\\ |u| \geq 1} } \prod_{j\in u} (\theta_j)  d_{-u}G(t) dt_u\\
&+ e^{\Theta \cdot t} \sum_{\emptyset\neq v \subseteq \{1,\dots,N\}}(-1)^{|v|}\sum_{w\subseteq -v} \prod_{j\in v+w} (\theta_j) d_{-v-w} G(t) dt_{v+w}.
\end{split}
\end{equation}
We complete the proof by showing that the sum of the last two terms above is zero. For this, let $u\subseteq \{1,\dots,N\}$ such that $|u| \geq 1$ be fixed and set $v+w = u$. Then the total number of terms
$$\prod_{j\in u} (\theta_j)  d_{-u}G(t) dt_u$$
in \eqref{kukkuukukkuu} is 
\begin{equation*}
\sum_{m=0}^{|u|} (-1)^m\binom{|u|}{m} = 0
\end{equation*}
by Lemma \ref{binomial}. This completes the proof.
\end{proof}
\end{lemma}

\begin{lemma}
\label{lemma:G}
Let $\Theta\in (0,\infty)^N$ and $Y = \{Y(e^t)\}_{t\in\mathbb{R}^N}$ be $\Theta$-self-similar. Then 
\begin{equation}
\label{eq:G}
G(t) = \begin{cases} 
      \int_T e^{-\Theta\cdot u} dY(e^u), & \sum_{l=1}^N t_l \geq 0 \\
      (-1)^N\int_T e^{-\Theta\cdot u} dY(e^u), & \sum_{l=1}^N t_l < 0, 
   \end{cases}
\end{equation}
where $T=T(t)$ is as in \eqref{T} and \eqref{T2}, belongs to $\mathcal{G}_{\Theta,0}$. Moreover, $dG(t) = e^{-\Theta\cdot t} d Y(e^t)$ in the multiple Riemann-Stieltjes sense for every integrand $g$ such that $g_t$ is continuous.
\begin{proof}
By the definition of the integral,  $G(t) = 0$ for every $t$ such that $\sum_l t_l=0$. Moreover, from \eqref{T>0} and \eqref{superintegral2},
\begin{equation}
\label{t>0}
\begin{split}
G(t)
&= e^{-\Theta\cdot t} Y(e^t)-\frac{1}{N} \sum_{|v| =1}Y(e^{t_{-v}:\tilde{t}_v})e^{-\Theta\cdot ({t_{-v}:\tilde{t}_v})} \\ &\quad + \sum_{\emptyset\neq v \subseteq \{1,\dots,N\}} \prod_{j\in v}(\theta_j) \int_{T_v} Y(e^{t_{-v}:u_v})e^{-\Theta\cdot ({t_{-v}:u_v})}du_v
\end{split}
\end{equation}
for every $\sum t\geq 0$ and 
\begin{equation}
\label{t<0}
\begin{split}
 G(t)
&= e^{-\Theta\cdot t} Y(e^t)-\frac{1}{N} \sum_{|v| =1} Y(e^{t_{-v}:\tilde{t}_v})e^{-\Theta\cdot ({t_{-v}:\tilde{t}_v})} \\ &\quad + \sum_{\emptyset\neq v \subseteq \{1,\dots,N\}}(-1)^{|v|} \prod_{j\in v}(\theta_j) \int_{T_v} Y(e^{t_{-v}:u_v})e^{-\Theta\cdot ({t_{-v}:u_v})}du_v   
\end{split}
\end{equation}
for every $\sum t < 0$. First, note that since $Y(e^{t_{-v}:\tilde{t}_v})e^{-\Theta\cdot ({t_{-v}:\tilde{t}_v})}$ depends only on $N-1$ of the variables, we obtain that 
\begin{equation}
\label{zerodifferential}
d Y(e^{t_{-v}:\tilde{t}_v})e^{-\Theta\cdot ({t_{-v}:\tilde{t}_v})}=0.
\end{equation}
Without loss of generality, we assume that $v=\{1,\dots,|v|\}$. In addition, we denote 
$$Y(e^{t_{-v}:u_v})e^{-\Theta\cdot ({t_{-v}:u_v})} = f(t_{-v}:u_v).$$
Then, by \eqref{T_v} for $\sum t \geq 0$,
\begin{equation*}
\begin{split}
&\int_{T_v} Y(e^{t_{-v}:u_v})e^{-\Theta\cdot ({t_{-v}:u_v})}du_v \\ &=\int_{-t_2-\dots-t_{N}}^{t_1} \int_{-u_1-t_3-\dots-t_{N}}^{t_2} \dots \int_{-u_1-\dots-u_{|v|-1}-t_{|v|+1}-\dots-t_N}^{t_{|v|}} f({t_{-v}:u_v}) du_v \eqqcolon F(t).
\end{split}
\end{equation*}
Similarly, by \eqref{Tv2} for $\sum t < 0$,
\begin{equation*}
\begin{split}
&(-1)^{|v|} \int_{T_v} Y(e^{t_{-v}:u_v})e^{-\Theta\cdot ({t_{-v}:u_v})}du_v \\ &=(-1)^{|v|}\int_{t_1}^{-t_2-\dots-t_{N}} \int_{t_2}^{-u_1-t_3-\dots-t_{N}} \dots \int_{t_{|v|}}^{-u_1-\dots-u_{|v|-1}-t_{|v|+1}-\dots-t_N} f({t_{-v}:u_v}) du_v \\
&= (-1)^{|v|}(-1)^{|v|}\int_{-t_2-\dots-t_{N}}^{t_1} \int_{-u_1-t_3-\dots-t_{N}}^{t_2} \dots \int_{-u_1-\dots-u_{|v|-1}-t_{|v|+1}-\dots-t_N}^{t_{|v|}} f({t_{-v}:u_v}) du_v\\
&= F(t).
\end{split}
\end{equation*}
In order to apply this observation to \eqref{t>0} and \eqref{t<0}, we first want to compute the mixed partial derivative $F_{t_v}(t)$. Denote
$$F(t) = \int_{-t_2-\dots-t_N}^{t_1} F^{(1)}(t_{\{2,\dots,N\}}: u_1) du_1.$$
Then
\begin{equation*}
\begin{split}
 &\frac{d}{dt_1} F(t) = F^{(1)} (t) \\
 &= \int_{-t_1-t_3-\dots -t_N}^{t_2} \int_{-t_1-u_2-t_4-\dots -t_N}^{t_3} \dots\\
 &\quad\int_{-t_1 - u_2 -\dots -u_{|v|-1} - t_{|v|+1}-\dots- t_N}^{t_{|v|}} f(t_{-v}:t_1:u_{\{2,\dots,|v|\}})  du_{\{3,\dots,|v|\}}du_2\\
 & \eqqcolon \int_{-t_1-t_3-\dots -t_N}^{t_2} F^{(2)}(t_{\{1,3,4,\dots,N\}}:u_2) du_2.
 \end{split}
 \end{equation*}
 Hence,
 \begin{equation*}
 \begin{split}
 &\frac{d^2}{dt_1dt_2} F(t) = \frac{d}{dt_2} F^{(1)}(t) = F^{(2)} (t)\\
 &= \int_{-t_1-t_2-t_4-\dots-t_N}^{t_3} \int_{-t_1-t_2-u_3-t_5-\dots -t_N}^{t_4} \dots\\
 &\quad \int_{-t_1-t_2-u_3-\dots -u_{|v|-1}-t_{|v|+1} - \dots -t_N}^{t_{|v|}} f(t_{-v}: t_1:t_2: u_{\{3,\dots,|v|\}}) du_{\{3,\dots,|v|\}}.
 \end{split}
 \end{equation*}
 By continuing similarly, we obtain 
 $$\frac{d}{dt_v} F(t) = \frac{d}{dt_{|v|}} \int_{-t_1-\dots-t_N}^{t_{|v|}} f(t_{-v}:t_{\{1,\dots,|v|-1\}}:u_{|v|}) du_{|v|}= f(t).$$
 Now, by Theorem \ref{theorem:radon2}, 
 $$dF(t) = d_{-v}F_{t_v}(t)dt_v = d_{-v} f(t) dt_v = d_{-v} \left(Y(e^t) e^{-\Theta\cdot t} \right) dt_v.$$
By combining with \eqref{zerodifferential} and then applying Theorem \ref{theorem:productrule}, we get
\begin{equation*}
\begin{split}
dG(t) &= d\left(e^{-\Theta\cdot t} Y(e^t)\right) + \sum_{\emptyset\neq v \subseteq \{1,\dots,N\}}\prod_{j\in v}(\theta_j) d_{-v} \left(Y(e^t) e^{-\Theta\cdot t} \right) dt_v\\
&= \sum_{\substack{v\subseteq \{1,\dots,N\}}}\prod_{j\in v}(\theta_j) d_{-v} \left(Y(e^t) e^{-\Theta\cdot t} \right) dt_v\\
&= \sum_{\substack{v\subseteq \{1,\dots,N\}}}\prod_{j\in v}(\theta_j) \sum_{w\subseteq -v} \prod_{j\in w}(-\theta_j) e^{-\Theta\cdot t} d_{-v-w} Y(e^t) dt_{v+w}\\
&= e^{-\Theta\cdot t}\sum_{\substack{v\subseteq \{1,\dots,N\}}} \sum_{w\subseteq -v} \prod_{j\in v+ w}(\theta_j)(-1)^{|w|} d_{-v-w} Y(e^t) dt_{v+w}.
\end{split}
\end{equation*}
When $v=w=\emptyset$, we get from above $e^{-\Theta\cdot t} dY(e^t)$. Note that for $v+w =q \neq \emptyset$, we observe by Lemma \ref{binomial} as before that the terms  $\prod_{j\in q} (\theta_j) d_{-q} Y(e^t) dt_q$ vanish.
Hence, we have shown that $dG(t) = e^{-\Theta\cdot t} dY(e^t)$. With this at hand, we can write
\begin{equation*}
 [G]_{s+h}^{t+h} = \int_{s+h}^{t+h} e^{-\Theta\cdot u} dY(e^u),
\end{equation*}
from which the change of variable $r = u-h$ and $\Theta$-self-similarity of $Y$ gives
\begin{equation*}
\begin{split}
 [G]_{s+h}^{t+h}&= \int_{s}^{t} e^{-\Theta\cdot (r+h)} dY(e^{r+h}) \overset{\text{law}}{=}  \int_{s}^{t} e^{\Theta\cdot h}e^{-\Theta\cdot (r+h)} dY(e^r)\\
 &= \int_{s}^{t}e^{-\Theta\cdot r} dY(e^r) = [G]_s^t.
\end{split}
\end{equation*}
The equality of laws can be verified by applying the self-similarity property of $Y$ to Riemann-Stieltjes sums, and passing to the limit. As multidimensional distributions can be treated similarly, we obtain that $G$ has stationary increments.

Finally, we show that 
$\int_{-\infty}^{t} e^{\Theta\cdot u} dG(r)$ is a well-defined random variable for all $t$. We have that
\begin{equation}
\label{sidedinverse}
 \int_{-\infty}^{t} e^{\Theta\cdot r} dG(r) =  \int_{-\infty}^{t}  dY(e^r) = \lim_{u\to\infty}[Y(e^r)]_{-\bar{u}}^t= Y(e^t),
 \end{equation}
since the values of $Y(e^r)$ at the other corner points except at $t$ tend to zero in probability by self-similarity of $Y$.
\end{proof}
\end{lemma}

\begin{theorem}
\label{theorem:bijection}
Let $\Theta\in (0,\infty)^N$. Let $G=\{G(t)\}_{t\in\mathbb{R}^N}\in\mathcal{G}_{\Theta,0}$ and $Y=\{Y(e^t)\}_{t\in\mathbb{R}^N}$ be $\Theta$-self-similar. The transformation
 $$(\mathcal{M}_\Theta Y)(t) \coloneqq \begin{cases} 
      \int_T e^{-\Theta\cdot u} dY(e^u), & \sum_{l=1}^N t_l \geq 0 \\
      (-1)^N\int_T e^{-\Theta\cdot u} dY(e^u), & \sum_{l=1}^N t_l < 0 
   \end{cases}$$
together with its inverse
$$(\mathcal{M}_\Theta^{-1} G)(e^t) \coloneqq \int_{-\infty}^{t}e^{\Theta\cdot u} dG(u)$$
define a bijection between $\Theta$-self-similar fields and fields in $\mathcal{G}_{\Theta,0}$ equipped with the equivalence relation of Definition \ref{defi:GH}.
\begin{proof}
By Lemmas \ref{lemma:Y} and \ref{lemma:G}, it suffices to show that 
$$\mathcal{M}_\Theta \mathcal{M}^{-1}_\Theta G = G \quad\text{and}\quad \mathcal{M}^{-1}_\Theta \mathcal{M}_\Theta  Y = Y.$$ 
Furthermore, Equation \eqref{sidedinverse} of Lemma \ref{lemma:G} gives directly that $\mathcal{M}^{-1}_\Theta \mathcal{M}_\Theta  Y = Y$.
For the remaining claim, by Lemma \ref{lemma:Y}, 
$$Y(e^t) = (\mathcal{M}_\Theta^{-1} G)(e^t)$$
is $\Theta$-self-similar with $dY(e^t) = e^{\Theta\cdot t} dG(t)$. Moreover, by Lemma \ref{lemma:G}, $\tilde{G}(t) \coloneqq (\mathcal{M}_\Theta Y)(t)$ is in $\mathcal{G}_{\Theta,0}$ with 
$$d\tilde{G}(t) = e^{-\Theta \cdot t} dY(e^t) = dG(t).$$
Denote $F  = G-\tilde{G}.$ Then $F(t) = 0$ for all $t$ such that $\sum_l t_l = 0.$ In addition, for every $t$ and $s$,
$$\left[F\right]_s^t = \left[G\right]_s^t - [\tilde{G}]_s^t = \int_s^t dG(t) - \int_s^t d\tilde{G}(t) = 0.$$
This gives
$$[F]_0^t = \sum_{v\subseteq \{1,\dots,N\}} (-1)^{|v|} F(0_v:t_{-v}) = \sum_{\substack{v\subseteq \{1,\dots,N\}\\
|v| \neq N}} (-1)^{|v|} F(0_v:t_{-v})= 0$$
and consequently,
$$F(t) = (-1)^{|v|+1} \sum_{\substack{\emptyset \neq v\subseteq \{1,\dots,N\}\\
|v| \neq N}} F(0_v:t_{-v}) \eqqcolon \sum_{\substack{\emptyset \neq v\subseteq \{1,\dots,N\}\\
|v| \neq N}} f^{(v)}(t_v).$$
That is, $\tilde{G} = G.$ This completes the proof.
\end{proof}
\end{theorem}

\begin{kor}
\label{cor:bijection}
Let $\mathcal{L}_\Theta$ be the Lamperti transformation  given in Definition \ref{defi:lamperti} that maps from stationary fields to $\Theta$-self-similar fields. Then $\mathcal{M}_\Theta \circ \mathcal{L}_\Theta$ maps from stationary fields to $\mathcal{G}_{\Theta,0}$ and $\mathcal{L}_\Theta^{-1} \circ \mathcal{M}_\Theta^{-1}$ maps from $\mathcal{G}_\Theta$ to stationary fields. Moreover, $\mathcal{M}_\Theta \circ \mathcal{L}_\Theta$ is a bijection between stationary fields and fields in $\mathcal{G}_{\Theta,0}$.
\end{kor}

\begin{theorem}
\label{theorem:1}
Let $\Theta \in(0, \infty)^{N}$ and let $X = \{X(t)\}_{t\in\mathbb{R}^N}$ be stationary. Then $X$ satisfies the Langevin equation \eqref{langevinequation} for $G = ( \mathcal{M}_\Theta \circ \mathcal{L}_\Theta)(X)\in\mathcal{G}_{\Theta,0}$.

\begin{proof}
By the Lamperti theorem, $X(t)=(\mathcal{L}_\Theta^{-1} Y)(t) = e^{-\Theta \cdot t} Y(e^{t})$ for a unique $\Theta$-self-similar $Y$. Denote $f(t)=e^{-\Theta \cdot t}$ and $Z(t)=Y(e^{t})$. By Theorem \ref{theorem:productrule}, the Langevin equation may be written as 
\begin{equation}
\label{modlangevin}
\begin{split}
dG(t) &= \sum_{u\subseteq\{1,\dots,N\}} \prod _{i\in u}\theta_i d_{-u}X(t) dt_u =\sum_{u\subseteq\{1,\dots,N\}} \prod _{i\in u}\theta_i d_{-u}(f(t)Z(t)) dt_u \\
&= \sum_{u\subseteq\{1,\dots,N\}} (-1)^{|u|}d_{-u}(f_{t_u}(t)Z(t)) dt_u\\
&= \sum_{u\subseteq\{1,\dots,N\}} (-1)^{|u|}\sum_{v \subseteq -u}f_{t_{u+v}}(t)d_{-u-v}Z(t) dt_vdt_u\\
&= \sum_{u\subseteq\{1,\dots,N\}} (-1)^{|u|}\sum_{v \subseteq -u}f_{t_{u+v}}(t)d_{-u-v}Z(t) dt_{u+v}.
\end{split}
\end{equation}
If $u=v=\emptyset$, then the corresponding term above is $f(t)dZ(t)$. Next, let $u+v = q \neq \emptyset$ be fixed. Now, $v = q-u$ and \eqref{modlangevin} turns to
\begin{equation}
\label{fixedsum}
\sum_{u\subseteq\{1,\dots,N\}} (-1)^{|u|}f_{t_{q}}(t)d_{-q}Z(t) dt_{q}.
\end{equation}
As before, applying Lemma \ref{binomial} one sees that the terms $f_{t_{q}}(t)d_{-q}Z(t) dt_{q}$ vanish and as such,  \eqref{modlangevin} reads
\begin{equation*}
f(t) d Z(t) =e^{-\Theta \cdot t} dY(e^{t}) = dG(t). 
\end{equation*}
Defining $G\in\mathcal{G}_{\Theta,0}$ according to Lemma \ref{lemma:G} completes the proof.
\end{proof}
\end{theorem}

\begin{theorem}
\label{theorem:2}
Let $\Theta \in(0, \infty)^{N}$ and $G \in \mathcal{G}_{\Theta}$. Then, the corresponding Langevin equation \eqref{langevinequation} has a unique stationary solution
$$X(t) = \left((\mathcal{L}_\Theta^{-1} \circ \mathcal{M}_\Theta^{-1} )(G)\right)(t) =  e^{-\Theta\cdot t}\int_{-\infty}^{t} e^{\Theta\cdot u} dG(u).$$
\begin{proof}
First, the field $X$ is stationary, since $X = (\mathcal{L}_\Theta^{-1} \circ \mathcal{M}_\Theta^{-1} )(G).$ In addition, by Theorem \ref{theorem:1}, $X$ satisfies the Langevin equation for $ \tilde{G} = ( \mathcal{M}_\Theta \circ \mathcal{L}_\Theta)(X)\in \mathcal{G}_{\Theta,0}$. By denoting $\mathcal{L}_\Theta(X) = Y = \mathcal{M}_\Theta^{-1}(G)$, we may write $\tilde{G} = \mathcal{M}_\Theta(Y).$ Now, as in the proof of Theorem \ref{theorem:bijection}, we get $d\tilde{G}(t) = dG(t).$ Hence, $X$ is a solution to the Langevin equation also for $G$.

Next, we prove the uniqueness. Assume that $\tilde{X}$ is another stationary solution. Then, $\tilde{X}$ satisfies the Langevin equation also for $\tilde{G} = ( \mathcal{M}_\Theta \circ \mathcal{L}_\Theta)(\tilde{X})\in\mathcal{G}_{\Theta,0}$. Thus, $d\tilde{G}(t) = dG(t)$ implying $\tilde{X} = (\mathcal{L}_\Theta^{-1} \circ \mathcal{M}_\Theta^{-1} )(\tilde{G}) = (\mathcal{L}_\Theta^{-1} \circ \mathcal{M}_\Theta^{-1} )(G) = X$ by the definition of the transformation $\mathcal{M}^{-1}$.
\end{proof}
\end{theorem}

We end the paper by considering the uniqueness of the solution to the Langevin equation. It turns out that the solution is unique up to certain lower-dimensional components that can be eliminated by posing appropriate ''boundary conditions''.
\begin{defi}
\label{defi:uniqueness}
    We say that fields $X(t)$ and $Y(t)$ are equivalent with respect to the parameter $\Theta$ if the field
    $$
    Z(t) = e^{\Theta \cdot t}[X(t)- Y(t)] 
    $$
    satisfies $[Z]_s^t = 0$ almost surely for all $s,t\in \mathbb{R}^N$.
\end{defi}
The above definition means that we have
\begin{equation}
\label{eq:XY-link}
X(t) = Y(t) + e^{-\Theta \cdot t}\sum_{\substack{u\subseteq \{1,\dots,N\}\\
|u| \neq N}}H^{(u)}(t_u)
\end{equation}
for some fields$H^{(u)}(t_u)$, where now $u$ is a proper subset of $\{1,\ldots,N\}$. The choice $u = \emptyset$ corresponds to the constant random variable $H^{(\emptyset)} (t_\emptyset) =C$. Note that the above definition is natural, as we cannot distinguish fields $H^{(u)}(t_u)$ in the multiple Riemann-Stieltjes sense. Indeed, their rectangular increments equal zero, since each $H^{(u)}(t_u)$ does not depend on all the variables, and are thus essentially lower-dimensional objects.

\begin{lemma}
\label{lemma:uniqueness}
The Langevin equation \eqref{langevinequation} admits a unique solution in the sense of Definition \ref{defi:uniqueness}.
\end{lemma}
\begin{proof}
Suppose that $X$ and $\tilde{X}$ are two solutions and set $Z= X-\tilde{X}$. Then, from the linearity of the equation, we obtain that $Z$ solves
\begin{equation*}
\sum_{u\subseteq\{1,\dots,N\}} \prod _{i\in u}\theta_i d_{-u}Z(t) dt_u = 0.
\end{equation*}
Multiplying with $e^{\Theta\cdot t}$ yields
\begin{equation*}
\begin{split}
\sum_{u\subseteq\{1,\dots,N\}} \prod _{i\in u}\theta_ie^{\Theta\cdot t} d_{-u}Z(t) dt_u &= \sum_{u\subseteq\{1,\dots,N\}} \frac{d e^{\Theta\cdot t}}{dt_u} d_{-u}Z(t) dt_u = d(e^{\Theta\cdot t}Z(t)) = 0
\end{split}
\end{equation*}
by Theorem \ref{theorem:productrule} under the standing assumption of almost surely continuous realizations. This gives $[e^{\Theta \cdot r}Z(r)]_s^t = 0$ for all $s,t$, and as in the proof of Theorem \ref{theorem:bijection}, 
\begin{equation*}
e^{\Theta\cdot t}Z(t) = \sum_{\substack{u\subseteq \{1,\dots,N\}\\
|u| \neq N}}H^{(u)}(t_u).
\end{equation*}
That is, the solution is unique in the sense of Definition \ref{defi:uniqueness}, see \eqref{eq:XY-link}. This completes the proof.
\end{proof}

\begin{rem}
By fixing some (random) ''boundary conditions'', we may eliminate some of the fields $H^{(u)}(t_u)$. For example, if in the two-dimensional setting $X(0,t_2) = X_2(t_2)$, we obtain that
$$H^{(2)}(t_2) = -H^{(1)}(0) - C.$$
From which,
$$X(t_1,t_2) =Y(t_1,t_2) + e^{-\Theta \cdot t} \left(H^{(1)}(t_1) - H^{(1)}(0)\right)=Y(t_1,t_2) + e^{-\Theta \cdot t} \left(H^{(1)}(t_1)  +\tilde{C}\right).$$
Note that, in the one-dimensional case, we would have
$$
X(t) = Y(t) + Ce^{-\theta t}
$$
from which the condition $X(0)=Y(0)$ would give uniqueness in the classical sense, i.e. $X(t)=Y(t)$ almost surely for all $t$.
\end{rem}



\begin{theorem}
\label{theorem:supreme}
Let $\Theta\in(0,\infty)^N$. A field $X = \{X(t)\}_{t\in\mathbb{R}^N}$ is stationary if and only if it is the unique solution to the Langevin equation \eqref{langevinequation} for $G\in\mathcal{G}_{\Theta,0}$ with the following ''boundary condition'': $X(t) = e^{-\Theta \cdot t}\int_{-\infty}^t e^{\Theta \cdot u} dG(u)$ for all $t$ such that $t_l=0$ at least for one $l\in\{1,\dots,N\}$. That is, 
$$X(t) = \left((\mathcal{L}_\Theta^{-1} \circ \mathcal{M}_\Theta^{-1} )(G)\right)(t) =  e^{-\Theta\cdot t}\int_{-\infty}^{t} e^{\Theta\cdot u} dG(u)\quad\text{for all } t\in\mathbb{R}^N.$$
Moreover, the field $G = \{G(t)\}_{t\in\mathbb{R}^N}$ in the representation is unique and expressible as
$$G(t) = \left(( \mathcal{M}_\Theta \circ \mathcal{L}_\Theta)(X)\right)(t).$$
\begin{proof}
Theorem \ref{theorem:1} shows that stationary fields solve the Langevin equation with the given boundary condition. Moreover, Theorem \ref{theorem:2} shows that there exists a stationary solution that satisfies the boundary condition. Hence, it remains to show that the solution under the boundary condition is unique. Based on the proof of Lemma \ref{lemma:uniqueness}, let 
\begin{equation*}
W(t) = e^{\Theta\cdot t}Z(t) = e^{\Theta\cdot t}(X(t)- \tilde{X}(t)) =  \sum_{\substack{u\subseteq \{1,\dots,N\}\\
|u| \neq N}}H^{(u)}(t_u). 
\end{equation*}
Let $0$ be the zero-vector of length $N$. Then, by the boundary condition, $W(0_v:t_{-v}) = 0$ for all $v\neq \emptyset$ and $t_{-v}\in\mathbb{R}^{|-v|}$. Thus,
\begin{equation*}
\begin{split}
&\sum_{\emptyset\neq v\subseteq \{1,\dots,N\}} (-1)^{|v|+1} W(0_v:t_{-v})\\ &= \sum_{\emptyset\neq v \subseteq \{1,\dots,N\}} (-1)^{|v|+1}\sum_{\substack{u\subseteq \{1,\dots,N\}\\
|u| \neq N}}H^{(u)}((0_v:t_{-v})_u)\\
&= \sum_{\substack{u\subseteq \{1,\dots,N\}\\
|u| \neq N}} \sum_{\emptyset\neq v \subseteq \{1,\dots,N\}}(-1)^{|v|+1}H^{(u)}((0_v:t_{-v})_u)=0.
\end{split}
\end{equation*}
Let $u\subset \{1,\dots,N\}$ be fixed above. This also fixes the field $H^{(u)}$. In the inner sum, $v$ and $v'$ are equivalent under $u$ if $(0_v:t_{-v})_u = (0_{v'}:t_{-v'})_u = q$, where $q$ is some $|u|$-vector consisting of zeros and the elements of $t\in\mathbb{R}^N$. Thus, the inner sum can be written as
\begin{equation*}
 \sum_{\emptyset\neq v \subseteq \{1,\dots,N\}}(-1)^{|v|+1}H^{(u)}((0_v:t_{-v})_u) = \sum_q H^{(u)}(q) \sum_{\substack{\emptyset\neq v \subseteq \{1,\dots,N\}\\
(0_v:t_{-v})_u = q}}(-1)^{|v|+1}.
\end{equation*}
Let $n$ be the number of zeros in $q$. If $n \geq 1$, then since the elements corresponding to $u$ are fixed and the elements corresponding to $-u$ can be either zeros or elements of $t$, we obtain that
\begin{equation*}
\sum_{\substack{\emptyset\neq v \subseteq \{1,\dots,N\}\\
(0_v:t_{-v})_u = q}}(-1)^{|v|+1} = (-1)^{n+1} \sum_{m=0}^{|-u|} (-1)^m\binom{|-u|}{m}= 0
\end{equation*}
by Lemma \ref{binomial}. If $n=0$, then since $v\neq \emptyset$, we get
\begin{equation*}
\sum_{\substack{\emptyset\neq v \subseteq \{1,\dots,N\}\\
(0_v:t_{-v})_u = q}}(-1)^{|v|+1} = \sum_{m=1}^{|-u|} (-1)^{m+1}\binom{|-u|}{m}=1+ (-1) \sum_{m=0}^{|-u|} (-1)^{m}\binom{|-u|}{m}=1.
\end{equation*}
Note that the above also holds in the case of $u = \emptyset$. Combining the previous results yields that
\begin{equation*}
\sum_{\emptyset\neq v\subseteq \{1,\dots,N\}} (-1)^{|v|+1} W(0_v:t_{-v}) = \sum_{\substack{u\subseteq \{1,\dots,N\}\\
|u| \neq N}}H^{(u)}(t_u) = W(t) = e^{\Theta\cdot t}(X(t)- \tilde{X}(t))
\end{equation*}
equals zero for all $t$, proving the uniqueness of the solution under the given boundary condition.

Lastly, we prove the uniqueness of $G$ in the representation. Assume that stationary $X$ satisfies the Langevin equation for $G\in\mathcal{G}_{\Theta,0}$ and $\tilde{G}\in\mathcal{G}_{\Theta,0}$. Then
$$X = (\mathcal{L}_\Theta^{-1} \circ \mathcal{M}_\Theta^{-1} )(G) = (\mathcal{L}_\Theta^{-1} \circ \mathcal{M}_\Theta^{-1} )(\tilde{G})$$
yielding $G = \tilde{G}$ by bijectivity of $\mathcal{L}_\Theta^{-1} \circ \mathcal{M}_\Theta^{-1}$. This completes the proof.
\end{proof}
\end{theorem}

\bibliographystyle{plain}
\bibliography{pipliateekki}

@article{Lamperti-1962,
	Author = {Lamperti, John},
	Journal = {Transactions of the American Mathematical Society},
	Issn = {0002-9947},
	Mrclass = {60.40 (60.30)},
	Mrnumber = {0138128},
	Mrreviewer = {M. C. Pike},
	Pages = {62--78},
	Title = {Semi-stable stochastic processes},
	Volume = {104},
	Year = {1962}}

@article{Viitasaari-2016a,
	Author = {Viitasaari, Lauri},
	Coden = {SPLTDC},
	Fjournal = {Statistics \& Probability Letters},
	Issn = {0167-7152},
	Journal = {Statistics \& Probability Letters},
	Mrclass = {60G07 (60G10 60G18 60H10)},
	Mrnumber = {3498367},
	Pages = {45--53},
	Title = {Representation of stationary and stationary increment processes via {L}angevin equation and self-similar processes},
	Volume = {115},
	Year = {2016},
}

@article{genton2007self,
  title={Self-similarity and {L}amperti transformation for random fields},
  author={Genton, Marc G and Perrin, Olivier and Taqqu, Murad S},
  journal={Stochastic Models},
  volume={23},
  number={3},
  pages={397--411},
  year={2007},
  publisher={Taylor \& Francis}
}

@article{makogin2019gaussian,
  title={Gaussian multi-self-similar random fields with distinct stationary properties of their rectangular increments},
  author={Makogin, Vitalii and Mishura, Yuliya},
  journal={Stochastic Models},
  volume={35},
  number={4},
  pages={391--428},
  year={2019},
  publisher={Taylor \& Francis}
}

@phdthesis{prause2015sequential,
  title={Sequential Nonparametric Detection of High-Dimensional Signals under Dependent Noise},
  author={Prause, Annabel},
  year={2015},
  school={Aachen, Technische Hochschule}
}

@book{sard1963linear,
  title={Linear Approximation},
  series={Mathematical Surveys and Monographs},
  author={Sard, Arthur},
  volume={9},
  year={1963},
  publisher={American Mathematical Society}
}

@article{youngfourier,
  title={On multiple {F}ourier series},
  author={Young, William H},
  journal={Proceedings of the London Mathematical Society},
  volume={s2-11},
 number={1},
  pages={133--184},
  year={1913},
}

@article{hardy1906double,
  title={On double {F}ourier series and especially those which represent the double zeta-function with real and incommensurable parameters},
  author={Hardy, Godfrey H},
  journal = {The Quarterly Journal of Mathematics},
  volume={37},
  number={1},
  pages={53--79},
  year={1906}
}

@article{clarkson1933definitions,
  title={On definitions of bounded variation for functions of two variables},
  author={Clarkson, James A and Adams, C Raymond},
  journal={Transactions of the American Mathematical Society},
  volume={35},
  number={4},
  pages={824--854},
  year={1933},
  publisher={JSTOR}
}

@article{clarkson1933double,
  title={On double {R}iemann-{S}tieltjes integrals},
  author={Clarkson, James A},
  journal={Bulletin of the American Mathematical Society},
  volume={39},
  number={12},
  pages={929--936},
  year={1933},
  publisher={American Mathematical Society}
}

@article{young1917multiple,
  title={On multiple integration by parts and the second theorem of the mean},
  author={Young, William H},
  journal={Proceedings of the London Mathematical Society},
  volume={2},
  number={1},
  pages={273--293},
  year={1917},
  publisher={Oxford Academic}
}

@article{discrete,
  journal       = {Theory of Probability and Mathematical Statistics},
  author = {Voutilainen, Marko and Viitasaari, Lauri and Ilmonen, Pauliina},
  title = {On {L}amperti transformation and characterisations of discrete random fields},
  volume = {111},
   year = {2024},
    pages = {181--197}
}

@incollection{owen2005multidimensional,
  title={Multidimensional variation for quasi-{M}onte {C}arlo},
  author={Owen, Art B},
  booktitle={Contemporary Multivariate Analysis and Design of Experiments: In Celebration of Professor Kai-Tai Fang's 65th Birthday},
  pages={49--74},
  year={2005},
  publisher={World Scientific}
}

@article{voutilainen2017model,
  title={On model fitting and estimation of strictly stationary processes},
  author={Voutilainen, Marko and Viitasaari, Lauri and Ilmonen, Pauliina},
  journal={Modern Stochastics: Theory and Applications},
  volume={4},
  number={4},
  pages={381--406},
  year={2017},
  publisher={VTeX: Solutions for Science Publishing}
}

@article{voutilainen2021vector,
  title={Vector-valued generalized {O}rnstein--{U}hlenbeck processes: {P}roperties and parameter estimation},
  author={Voutilainen, Marko and Viitasaari, Lauri and Ilmonen, Pauliina and Torres, Soledad and Tudor, Ciprian},
  journal={Scandinavian Journal of Statistics},
  year={2022},
  volume = {49},
  number = {3},
  pages = {992-1022},
  publisher={Wiley Online Library}
}

@article{voutilainen2020modeling,
  title={Modeling and estimation of multivariate discrete and continuous time stationary processes},
  author={Voutilainen, Marko},
  journal={Frontiers in Applied Mathematics and Statistics},
  volume={6},
  year={2020},
  publisher={Frontiers}
}

@incollection{embrechts2002selfsimilar,
  title={Selfsimilar {P}rocesses},
  author={Embrechts, Paul},
  booktitle={Princeton Series in Applied Mathematics},
  year={2002},
  publisher={Princeton University Press}
}

@article{makogin2015example,
  title={Example of a {G}aussian self-similar field with stationary rectangular increments that is not a fractional {B}rownian sheet},
  author={Makogin, Vitalii and Mishura, Yuliya},
  journal={Stochastic Analysis and Applications},
  volume={33},
  number={3},
  pages={413--428},
  year={2015},
  publisher={Taylor \& Francis}
}

@book{hildebrandt1963introduction,
  title={Introduction to the Theory of Integration},
  series = {Pure and Applied Mathematics},
  author={Hildebrandt, Theophil H},
  volume={8},
  publisher={Academic Press},
  year={1963}
}

@article{terdik2005notes,
  title={Notes on fractional {O}rnstein--{U}hlenbeck random sheets},
  author={Terdik, Gy{\"o}rgy and Woyczynski, Wojbor A},
  journal={Publicationes Mathematicae Debrecen},
  volume={66},
  number={1/2},
  pages={153--181},
  year={2005}
}

@article{de2012least,
  title={Least squares estimator for the parameter of the fractional {O}rnstein--{U}hlenbeck sheet},
  author={De la Cerda, Jorge Clarke and Tudor, Ciprian A},
  journal={Journal of the Korean Statistical Society},
  volume={41},
  number={3},
  pages={341--350},
  year={2012},
  publisher={Elsevier}
}

@article{voutilainen2025one,
  title={One-to-one correspondences between discrete multivariate stationary, self-similar, and stationary increment fields},
  author={Voutilainen, Marko and Peltonen, Valtteri},
  journal={Stochastic Models},
  pages={1--34},
  year={2025},
  publisher={Taylor \& Francis}
}

@article{bierme2007operator,
  title={Operator scaling stable random fields},
  author={Bierm{\'e}, Hermine and Meerschaert, Mark M and Scheffler, Hans-Peter},
  journal={Stochastic Processes and their Applications},
  volume={117},
  number={3},
  pages={312--332},
  year={2007},
  publisher={Elsevier}
}
\end{document}